\documentclass[11pt,a4paper,english]{amsart}

\usepackage[margin=2.5cm,
]{geometry}

\usepackage{csquotes}
\AddToHook{bfseries}{\boldmath}

\usepackage{iftex}

\ifPDFTeX
    \usepackage[T1]{fontenc}
    \usepackage{lmodern}
\fi

\usepackage{newunicodechar}
\newunicodechar{ε}{\ensuremath{\varepsilon}}

\usepackage{thmtools}
\usepackage{amsmath}
\usepackage{amssymb}
\usepackage{bbm}
\usepackage{pdfsync}
\usepackage{esint} %
\usepackage{verbatim} %

\newcommand{\R}{{\mathbb R}}       %

\newcommand{\bS}{{\mathbb S}}

\newcommand{\Z}{{\mathbb Z}}       %
\newcommand{\DD}{{\mathcal D}}

\newcommand{\HH}{{\mathcal H}}

\newcommand{\cW}{{\mathcal W}}

\newcommand{\cN}{{\mathcal N}}

\newcommand{\EE}{{\mathcal E}}

\newcommand{\diam}{{\rm diam}}
\newcommand{\dist}{{\rm dist}}

\newcommand{\supp}{\operatorname{supp}}

\newcommand{\vphi}{{\varphi}}
\newcommand{\ve}{{\varepsilon}}
\newcommand{\vv}{{\vspace{2mm}}}

\newcommand{\wt}[1]{{\widetilde{#1}}}
\newcommand{\wh}[1]{{\widehat{#1}}}

\newcommand{\pom}{{\partial\Omega}}

\newcommand{\tree}{{\rm Tree}}

\newcommand{\dd}{{\mathrm{d}}}

\newcommand{\BA}{{\mathsf{BA}}}

\def\XXint#1#2#3{{\setbox0=\hbox{$#1{#2#3}{\int}$ }
\vcenter{\hbox{$#2#3$ }}\kern-.58\wd0}}

\def\Lip{\mathop\mathrm{Lip}} 	                       %
\def\H11{\textup{H}^1_1} 					%

\usepackage{pgf,tikz} %

\definecolor{ffffff}{rgb}{1.0,1.0,1.0}
\definecolor{qqqqff}{rgb}{0.0,0.0,1.0}
\definecolor{ffqqqq}{rgb}{1.0,0.0,0.0}
\definecolor{zzzzqq}{rgb}{0.6,0.6,0.0}
\definecolor{marronet}{rgb}{0.6,0.2,0}
\definecolor{negre}{rgb}{0,0,0}
\definecolor{vermell}{rgb}{0.8,0.05,0.05}
\definecolor{blau}{rgb}{0.3,0.2,1.}
\definecolor{blauclar}{rgb}{0.,0.,1.}
\definecolor{grisfosc}{rgb}{0.25098039215686274,0.25098039215686274,0.25098039215686274}
\definecolor{verd}{rgb}{0.1,0.6,0.1}
\definecolor{taronja}{rgb}{0.9,0.6,0.05}
\definecolor{vermellclar}{rgb}{1.,0.,0.}
\definecolor{verdet}{rgb}{0,0.8,0.1}
\definecolor{blauverd}{rgb}{0,0.4,0.2}
\definecolor{grisclar}{rgb}{0.6274509803921569,0.6274509803921569,0.6274509803921569}

\usepackage[
	backend=biber,
	style=alphabetic, %
	maxalphanames=4,
	minalphanames=3,
	giveninits=true,
]{biblatex}
\DeclareSourcemap{
	\maps[datatype=bibtex]{
    \map{
        \step[fieldsource=mrnumber, fieldtarget=eprint, final]
        \step[fieldset=eprinttype, fieldvalue=mr]
    }
		\map{
			\step[fieldsource=doi,final]
			\step[fieldset=isbn,null]
			\step[fieldset=url,null]
			\step[fieldset=urldate,null]
			\step[fieldset=issn,null]
			\step[fieldset=eprint,null]
		}
	}
}%

\DeclareFieldFormat{eprint:mr}{%
    \ifhyperref
        {\href{https://mathscinet.ams.org/mathscinet-getitem?mr=#1}{MR#1}}
        {MR#1}%
}

\usepackage{hyperref}
\hypersetup{hidelinks}
\usepackage{orcidlink}
\usepackage[nameinlink]{cleveref}

\newtheorem{theorem}{Theorem}[section]
\newtheorem{lemma}[theorem]{Lemma}

\newtheorem{coro}[theorem]{Corollary}

\newtheorem*{claim*}{Claim}

\newtheorem*{theorem*}{Theorem}
\newtheorem*{theoremA}{Theorem A}
\newtheorem*{theoremB}{Theorem B}
\newtheorem*{theoremC}{Theorem C}
\newtheorem*{theoremD}{Theorem D}
\newtheorem*{theoremE}{Theorem E}

\crefname{coro}{corollary}{corollaries}
\Crefname{coro}{Corollary}{Corollaries}

\theoremstyle{definition}

\theoremstyle{remark}
\newtheorem{rem}[theorem]{Remark}

\crefname{rem}{remark}{remarks}
\Crefname{rem}{Remark}{Remarks}

\numberwithin{equation}{section}

\newcommand{\brem}{\begin{rem}}
\newcommand{\erem}{\end{rem}}

\title{Characterization of $C^{1,\alpha}$-domains in terms of Carleson-type $\varepsilon$-coefficients}

\author[Xavier Tolsa]{Xavier Tolsa\,\orcidlink{0000-0001-7976-5433}}
\address{
  ICREA, 08010 Barcelona,  Catalonia; Departament de Matemàtiques,
  Universitat Autònoma de Barcelona, 08193 Bellaterra, Catalonia; and Centre de Recerca Matemàtica, 08193 Bellaterra, Catalonia.}
\email{\href{mailto:xavier.tolsa@uab.cat}{xavier.tolsa@uab.cat}}

\author[Tobias Wang]{Tobias Wang\,\orcidlink{0009-0007-8208-308X}}
\address{
  School of Mathematics,
  Georgia Institute of Technology,
  Atlanta,
  GA 30332,
  USA
}
\email{\href{mailto:tobias.wang@gatech.edu}{tobias.wang@gatech.edu}}

\date{October 14, 2026}

\hypersetup{
    pdftitle={Characterization of C(1,α)-domains in terms of Carleson-type ε-coefficients},
    pdfauthor={Xavier Tolsa, Tobias Wang},
    pdfkeywords={Carleson's coefficient, C(1,α)-domains, Hölder domains, Littlewood-Paley estimates, rectifiability.}
}

\begin{document}

\begin{abstract}
  Let $\Omega=\Omega^+\subset \mathbb{R}^{n+1}$ be a bounded open set with $n$-Ahlfors regular boundary and let $\Omega^-=\mathbb{R}^{n+1}\setminus\overline{\Omega^+}$. For $x\in\partial\Omega$, $r>0$, define
  \begin{equation*}
    \varepsilon_{n,\Omega}(x,r):=\frac{1}{r^n} \inf_{H^+}\mathcal{H}^n\left(((\partial B(x,r)\cap H^+)\setminus \Omega^+) \cup ((\partial B(x,r)\cap H^-)\setminus \Omega^-)\right),
  \end{equation*}
  where $\mathcal{H}^n$ stands for $n$-dimensional Hausdorff measure and
  the infimum is taken over all open affine half-spaces $H^+$
  such that $x\in \partial H^+$.
  Here $H^-=\mathbb{R}^{n+1}\setminus \overline{H^+}$.
  In this work we prove that $\Omega$ is a $C^{1,\alpha}$ open set, with $0<\alpha<1$, if and only if there exists some constant $C>0$ such that $\varepsilon_{n,\Omega}(x,r)\leq C\,r^\alpha$ for all $x\in\partial\Omega$ and all $0<r\leq \mathrm{diam}(\partial\Omega)$.
\end{abstract}

\maketitle

\section{Introduction}

The characterization of the smoothness of sets in Euclidean space
is a classical topic in analysis. In recent decades, the study of
problems in analysis, such as Vitushkin's conjecture
(see \cite{David-vitus}), has motivated the introduction of tools
from harmonic analysis into geometric measure theory.
For example, David and Semmes (see
\cite{DS1,DS2}) developed the theory of uniform rectifiability and characterized uniformly rectifiable sets in terms
of the behavior of $\beta$-type coefficients (see \cref{secprelim} for these notions). More recently, the study of the relationship between the so-called Carleson $\ve$-coefficients
and rectifiability has been a subject of intense investigation.

In the plane, consider an open set $\Omega^+$ and $\Omega^- = \R^2\setminus \overline{\Omega^+}$, $x\in\R^2$, and $r>0$.
Denote by $I^+(x,r)$ and $I^-(x,r)$ the longest open arcs of the circle
$\partial B(x,r)$ contained in $\Omega^+$ and $\Omega^-$, respectively (they may be empty). The Carleson $\ve$-coefficient
associated with $\Omega\equiv\Omega^+$ is defined as
\begin{equation}\label{eqepsiloncoef}
  \ve_\Omega(x,r) = \frac1r\,\max\big(\big|\pi r- \HH^1(I^+(x,r))\big|,\, \big|\pi r- \HH^1(I^-(x,r))\big|\big).
\end{equation}
In \cite{JTV} the following result, known as Carleson's $\ve^2$-conjecture, was shown:

\begin{theoremA}
  Let $\Omega=\Omega^+\subset \R^2$ be a Jordan domain, let $\Omega^-=\R^2\setminus \overline{\Omega^+}$, and let $\Gamma=
    \partial\Omega^+$.  Then, up to a set of zero $\HH^1$-measure,
  \begin{equation}\label{eqCar00}
    \mbox{$\displaystyle \EE_\Omega(x)^2:=\int_0^1 \ve_\Omega(x,r)^2\,\frac{dr}r<\infty$\; for $x\in\Gamma$ \quad$\Leftrightarrow$ \quad
      $\Omega$ has a tangent at $x\in\Gamma$.}
  \end{equation}
  In particular, the set
  $G=\{x\in\Gamma:\EE(x)<\infty\}$ is $1$-rectifiable.
\end{theoremA}

The fact  that $\EE_\Omega(x)<\infty$ for $\HH^1$-a.e.\ tangent point on a Jordan curve had been proved previously by Bishop in
\cite{Bishop-thesis} (see also \cite{BCGJ}). The most difficult implication of Theorem A, i.e., the fact that the set $G$ is $1$-rectifiable and tangents to $\Gamma$ exist for $\HH^1$-a.e.\ $x\in G$,
was proved more recently by Jaye, the first author of this paper, and Villa \cite{JTV}.

The preceding theorem has been extended to higher dimensions by Fleschler, Tolsa, and Villa \cite{FTV2} in terms of some variants of Carleson's coefficients.
Given two arbitrary disjoint open (or, more generally, Borel) sets $\Omega^+$, $\Omega^-$ $\subset \R^{n+1}$, and $x\in \mathbb{R}^{n+1}$, $r>0$, define
\begin{equation*}
  \varepsilon_{n,\Omega^\pm}(x,r):=\frac{1}{r^n} \inf_{H^+}\mathcal{H}^n\left(((\partial B(x,r)\cap H^+)\setminus \Omega^+) \cup ((\partial B(x,r)\cap H^-)\setminus \Omega^-)\right),
\end{equation*}
where the infimum is taken over all open affine half-spaces $H^+$
such that $x\in \partial H^+$.
Here $H^-=\mathbb{R}^{n+1}\setminus \overline{H^+}$.
For a fixed open set $\Omega\subset \R^{n+1}$, we consider the coefficient $\ve_{n,\Omega}(x,r)$ associated with the sets $\Omega^+=\Omega$ and $\Omega^-=\R^{n+1}\setminus \overline{\Omega}$. For a Jordan domain in the plane, it is easy to check that $\ve_{1,\Omega}(x,r)\lesssim
  \ve_\Omega(x,r)$ (see \cite{FTV1}).

Among other results, the following is proven in \cite{FTV2}:

\begin{theoremB}
  For $n\geq 1$, let $\Omega^+, \Omega^- \subset \R^{n+1}$ be two disjoint Borel subsets, and let
  $$E\subset\bigg\{x \in \R^{n+1}:  \int_0^1 \ve_{n,\Omega^\pm}(x,r)^2 \, \frac{dr}{r} < \infty\bigg\}$$
  be a Borel set. Then $E$ is $n$-rectifiable.
\end{theoremB}

We remark that, as shown in \cite{FTV2}, when $\Omega^+, \Omega^-$ are open, the set of points $x \in \R^{n+1}$ such that $\int_0^1 \ve_{n,\Omega^\pm}(x,r)^2 \, \frac{dr}{r} < \infty$ is Borel. Under the additional assumption that $\Omega^+ \cup \Omega^-$ satisfies the so-called capacity density condition, the $\HH^n$-a.e.\ existence of tangents is also characterized in \cite{FTV2} in terms of the finiteness of a square function involving somewhat stronger coefficients, denoted by $\alpha_{\Omega^\pm}(x,r)$ (see \cref{secprelim}).

For Ahlfors regular sets we have more quantitative results.
A Borel measure $\mu$ on $\mathbb{R}^{n+1}$ is said to be $n$-Ahlfors regular if there exists some constant $C>0$ such that
\begin{equation}\label{e:def-ADR}
  C^{-1}r^n\leq \mu(B(x,r))\leq Cr^n\qquad \text{for all}\, x\in \supp\mu, \quad 0<r\leq \diam(\supp\mu).
\end{equation}
A closed set $E\subset \mathbb{R}^{n+1}$ is called $n$-Ahlfors regular if the Hausdorff measure $\HH^n|_E$ is $n$-Ahlfors regular.
In \cite{CTV}, the following has been proven:

\begin{theoremC}
  Let $\Omega^+$ and $\Omega^-$ be two disjoint open subsets of $\R^{n+1}$. Suppose that $\mu$ is an $n$-Ahlfors regular measure with $\supp(\mu) = \partial \Omega^+ \cup \partial \Omega^-$. Then the following are equivalent.
  \begin{enumerate}
    \item $\Omega^+$ and $\Omega^-$ are complementary two-sided corkscrew open sets, and in particular $\mu$ is uniformly $n$-rectifiable.
    \item There is a constant $C_1$ so that for each ball $B$ centered on $\supp(\mu)$ it holds that
          \begin{equation*}%
            \int_B \int_0^{r(B)} \ve_{n,\Omega^\pm} (x,r)^2 \, \tfrac{dr}{r} \, d\mu(x) \leq C_1 r(B)^n.
          \end{equation*}
  \end{enumerate}
\end{theoremC}

For the definition of uniform rectifiability and the corkscrew condition, see \cref{secprelim}.
\vv

It is natural to ask whether some of these results have
counterparts that characterize higher degrees of smoothness.
In this direction, the following was proven by Casey \cite{Casey}:

\begin{theoremD}
  Let $\Omega=\Omega^+\subset \R^2$ be a Jordan domain, let $\Omega^-=\R^2\setminus \overline{\Omega^+}$, and let $\Gamma=
    \partial\Omega^+$.
  Let $\alpha>0$ and suppose that there exists some constant $C>0$ such that
  \begin{equation}\label{eqhyp91}
    \ve_\Omega(x,r)\leq C\,r^{\alpha}\quad \mbox{ for all $x\in\Gamma$ and all $0<r\leq \diam(\Gamma)$.}
  \end{equation}
  Then $\Omega$ is a $C^{1,\beta}$ domain, for some $\beta>0$ depending only on $\alpha$.
\end{theoremD}

In the converse direction, if $\Omega\subset\R^2$ is a $C^{1,\alpha}$ domain, then it is easy to check that
$\ve_\Omega(x,r)\leq C\,r^{\alpha}$ for all $x\in\Gamma$ and all $0<r\leq \diam(\Gamma)$.
A natural question is whether one can take $\beta=\alpha$
in Casey's theorem.
In the present paper we show that this is indeed the case. In fact, we prove the following result in arbitrary dimensions for open sets with Ahlfors regular boundaries:

\begin{theorem}\label{teomain}
  Let $\Omega\subset\R^{n+1}$ be a bounded open set with $n$-Ahlfors regular boundary and let $0<\alpha<1$. Then $\Omega$ is a $C^{1,\alpha}$ open set
  if and only if there exists some constant $C$ such that
  \begin{equation}\label{eqhyp92}
    \ve_{n,\Omega}(x,r)\leq C\,r^{\alpha}\quad \mbox{ for all $x\in\pom$ and all $0<r\leq \diam(\pom)$.}
  \end{equation}
\end{theorem}

We remark that one of the implications in the theorem is immediate: if $\Omega$ is a $C^{1,\alpha}$ domain, then
$\ve_n(x,r)\lesssim r^{\alpha}.$
The converse implication is the main contribution of our work.
In the plane, when $\Omega$ is a Jordan domain satisfying \cref{eqhyp91}, Casey's theorem implies that $\Omega$ is a $C^{1,\beta}$
domain and thus its boundary is $1$-Ahlfors regular. The condition \cref{eqhyp91} implies \cref{eqhyp92} (because $\ve_1(x,r)\lesssim\ve(x,r)$), and \cref{teomain} ensures that $\Omega$ is a $C^{1,\alpha}$ domain.

While working on the present paper, we learned that Casey, Engelstein, and Toro were working on a similar project, independently of us.
Very recently they published a preprint \cite{CET} where they prove the following:

\begin{theoremE}
  For $n\geq1$, let $\Omega^+, \Omega^-\subset\R^{n+1}$ be open and disjoint. For fixed $\alpha>0$ and $r_0>0$, suppose that there
  exists a constant $C>0$ such that
  \begin{equation}\label{eqCar00*}
    \alpha_{\Omega^\pm}(x,r)\leq C\,r^{2\alpha}\quad \mbox{ for all $x\in\pom^+\cup \pom^-$ and all $0<r\leq r_0$.}
  \end{equation}
  Then $\pom^+=\pom^-$ is a $C^{1,\beta}$ manifold, with $\beta>0$ depending on $\alpha$.
\end{theoremE}

For the precise definition of the coefficients $\alpha_{\Omega^\pm}$, see \cref{secprelim}. We remark that the theorem above
is stated in \cite{CET} in terms of different coefficients. However, as explained by the authors, their main result is equivalent to the one above.

It is worth comparing the main result of our work, \cref{teomain}, with the preceding theorem of Casey, Engelstein, and Toro. First we remark that \cref{teomain} is also valid for pairs of disjoint open sets $\Omega^+$, $\Omega^-$ (such that the union of their boundaries is Ahlfors regular). This follows easily by combining \cref{teomain} with Theorem C.
Further, \cref{teomain} gives the sharp regularity $C^{1,\alpha}$.
Another difference is that
Theorem E by Casey et al.\ is stated in terms of the stronger coefficients $\alpha_{\Omega^\pm}(x,r)$. Indeed,
as shown in \cite{FTV1}, it always holds that $\ve_{n,\Omega^\pm}(x,r)^2\lesssim \alpha_{\Omega^\pm}(x,r)$, while the opposite inequality may fail. This follows easily from the fact that the $\alpha$ coefficients detect sets of Hausdorff dimension $s>n-2$ in the $n$-dimensional sphere $S(x,r)$, while $\ve_n$ only detects sets of positive $\HH^n$-measure (see \cite{FTV1} for more details).
On the other hand, a very nice feature of Theorem E is that it
does not require the assumption of Ahlfors regularity, while \cref{teomain} does (in fact this is an important ingredient of our arguments).

The results of Casey, Engelstein, and Toro and ours were obtained almost simultaneously, and we believe that both are interesting in their own right. Finally, notice that by combining the two works, one easily deduces that, under the assumptions of Theorem E for $0<\alpha<1$,
one gets the sharp result that $\pom^+=\pom^-$ is a $C^{1,\alpha}$ manifold. In fact, once one knows that $\pom^+$ is a $C^{1,\beta}$ manifold, it turns out that this boundary is $n$-Ahlfors regular, and since $\ve_{n,\Omega^\pm}(x,r)^2\lesssim \alpha_{\Omega^\pm}(x,r)$, the assumptions of \cref{teomain} are satisfied.\footnote{The Ahlfors regularity of $\pom$ is clear if $\pom$ is bounded. When $\pom$ is unbounded, instead one could apply \cref{thm:Rn-unit-ball} below.}

The plan of the paper is as follows. \Cref{secprelim} contains some preliminary notation and background results.
In \cref{secplane} we consider the planar case of \cref{teomain}. In this case, the arguments are simpler and our theorem even extends to $\alpha=1$.
Turning to arbitrary dimensions in \cref{secLipschitgraph},
we prove a local result for the particular
case of Lipschitz domains with sufficiently small constant.
The arguments in \cref{secplane} use second differences,
while the argument in \cref{secLipschitgraph} uses the
characterization of $C^{1,\alpha}$ functions in terms of
Littlewood--Paley pieces.
The main difficulty is to extend and adapt these methods to
Lipschitz domains and radial square functions in $\R^{n+1}$.

The general Ahlfors regular case of \cref{teomain} is treated
in \cref{secgeneral,sec-Lipgraph,secfinal}.
The strategy is to approximate $\pom$ locally by a Lipschitz
graph, transfer a suitable variant of the estimates for the
$\ve^2$-function of $\Omega$ to that graph, and then prove that
its slope is very small. This allows us to apply the previous
result for Lipschitz domains with small Lipschitz constant.

\subsection*{Acknowledgements.}
X.T. was  partially supported by MICIU (Spain) under the grant PID2024-160507NB-I00.
T.W. was supported in part by the Department of Education's Graduate Assistance in Areas of National Need Award (P200A240169).
Part of this work was carried out during a visit of the second author to the Universitat Aut\`onoma de Barcelona.
T.W. thanks UAB for its hospitality and Ben Jaye for supporting the visit.

\subsection*{AI disclosure}
The results of this paper and our original proofs were
obtained without the use of AI tools.
We subsequently used ChatGPT to review a preliminary
version of the manuscript. This review suggested two
simplifications of arguments that we had already established.

The first concerns \cref{thm:bk-upper-bound-Aj},
which we originally formulated using a square function
involving the second differences
$\delta_h^2f(x)=f(x+h)+f(x-h)-2f(x)$.
The second concerns the comparison argument in
\cref{lemmaguai*}, used to prove
\cref{lem7.11}.
Our original proof introduced auxiliary symmetry coefficients,
in the spirit of \cite{FTV2}, and used
\cite[Lemma 14.12]{FTV2}.
The current argument avoids these coefficients and the use
of that lemma, making the proof more self-contained.
The estimates and their presentation in the manuscript
differ from those initially suggested by ChatGPT.

\vv

\section{Preliminaries}\label{secprelim}
In the paper, constants denoted by $C$ or $c$ depend just on the dimension unless otherwise stated. We will write $a\lesssim b$ if there is $C>0$ such that $a\leq Cb$. We write $a\approx b$ if $a\lesssim b\lesssim a$.

Open balls in $\R^{n+1}$ centered at $x$ with radius $r>0$ are denoted by $B(x,r)$, and closed balls by
$\bar B(x,r)$. For an open or closed ball $B\subset\R^{n+1}$ with radius $r$, we write $r(B)=r$.
An open annulus in $\R^{n+1}$ centered at $x$ with inner radius $r_1$ and outer radius $r_2$ is denoted by
$A(x,r_1,r_2)$, and the corresponding closed annulus by $\bar A(x,r_1,r_2)$.
We use the two notations $S(x,r)\equiv \partial B(x,r)$ for spheres
in $\R^{n+1}$ centered at $x$ with radius $r$.

\subsection{Uniform rectifiability,  the corkscrew condition,  and \texorpdfstring{$\alpha$}{α} type coefficients}
A set $E\subset \R^{n+1}$ is called $n$-rectifiable if there is a countable collection of Lipschitz maps $g_i:\R^n\to\R^{n+1}$
such that
$$\HH^n\Big(E\setminus \bigcup_i g_i(\R^n)\Big)=0.$$

We say that a Borel measure $\mu$ in $\R^{n+1}$ is uniformly $n$-rectifiable if it is $n$-Ahlfors regular and there exist $\theta,M>0$ such that, for all $x\in\supp\mu$
and $0<r\leq \diam(\supp\mu)$, there exists a Lipschitz mapping $g$ from the ball $B_n(0,r)\subset\R^n$ to $\R^{n+1}$ with
$\Lip(g)\leq M$ such that
$$\mu(B(x,r)\cap g(B_n(0,r)))\geq \theta\,r^n.$$
Also, a set $E\subset\R^{n+1}$ is called uniformly $n$-rectifiable
if $\HH^n|_E$ is uniformly $n$-rectifiable.

Let $\Omega\subset \mathbb{R}^{n+1}$ be an open set. We say that $\Omega$ satisfies the \textit{$c$-corkscrew condition} if there exists some $c>0$ such that for all $x\in \partial \Omega$ and $0<r\leq\mathrm{diam}(\partial\Omega)$ there exists some ball $B\subset \Omega \cap \overline{B(x,r)}$ with $r(B)\geq cr$.

Next, we say that $\Omega$ satisfies the two-sided $c$-corkscrew
condition if, for every $x\in\partial\Omega$ and
$0<r\leq\operatorname{diam}(\partial\Omega)$, there exist balls
$
  B^+\subset\Omega\cap B(x,r)$,
and
$B^-\subset
  (\R^{n+1}\setminus\overline\Omega)\cap B(x,r),
$
with $r(B^\pm)\geq cr$.

We next introduce the spectral $\alpha$ coefficients
associated with an open set $\Omega$.
Given an open set $V\subset \bS^n$,
where $\bS^n=
  S(0,1)$ is the unit sphere in $\R^{n+1}$, one says that $u\in W^{1,2}_0(V)$ is a Dirichlet eigenfunction of
$V$ for the Laplace-Beltrami operator $-\Delta_{\bS^n}$ if $u\not\equiv 0$ and
$$-\Delta_{\bS^{n}} u = \lambda\,u,$$
for some $\lambda\in\R$. Here $W^{1,2}_0(V)$ is the closure in the Sobolev space $W^{1,2}(V)$ of the $C^\infty$ functions  with compact support in $V$.
The number $\lambda$ is the eigenvalue associated with $u$. It is well known that all the eigenvalues of the Laplace-Beltrami operator are non-negative.
The characteristic constant of $V$ is defined as the non-negative number $\alpha_V$ such that $\lambda_V = \alpha_V(n-1+\alpha_V)$, where $\lambda_V$ is the smallest eigenvalue.

The Friedland-Hayman inequality \cite{FH} asserts that,
for any two disjoint open subsets $V_1,V_2\subset \bS^{n}$
and $\alpha_i\equiv\alpha_{V_i}$,
$$\alpha_1+\alpha_2-2\geq 0.$$
This inequality plays an essential role in the proof of
the Alt-Caffarelli-Friedman  monotonicity formula, which is a basic tool in many free boundary problems.

Given two disjoint open sets $\Omega_1,\Omega_2\subset \R^{n+1}$ and $x\in\R^{n+1}$, $r>0$, we consider
the re-scaled sets
$V_i(x,r) =\{r^{-1}(x-y) \, : \, y \in S(x,r)\cap \Omega_i\}$, and we denote
\begin{equation}\label{eqalfai}
  \alpha_i(x,r) = \alpha_{V_i(x,r)}.
\end{equation}
Notice that $\alpha_1(x,r) + \alpha_2(x,r) - 2\geq0$ by the Friedland-Hayman  inequality.
Then we define
$$\alpha_{\Omega_1,\Omega_2}(x,r) =  \min(1,\alpha_1(x,r) + \alpha_2(x,r) - 2).$$
For disjoint open sets $\Omega^+$, $\Omega^-$, we also use the notation $\alpha_{\Omega^\pm}(x,r)$.

In the plane, an easy computation (see \cite{AKN}) shows that the Carleson coefficient $\ve(x,r)$ associated with $\Omega_1$ and $\Omega_2$ satisfies
\begin{equation}\label{eqalfai2}
  \ve_{\Omega_1,\Omega_2}(x,r)^2\lesssim \alpha_{\Omega_1,\Omega_2}(x,r) .
\end{equation}
Further, if $I_1(x,r)$ and $I_2(x,r)$ are complementary arcs in $S(x,r)$, one can check that
$$\alpha_{\Omega_1,\Omega_2}(x,r) %
  \approx \ve(x,r)^2.$$

In higher dimensions we have
\begin{equation}\label{eqclau10}
  \ve_{n,\Omega_1,\Omega_2}(x,r)^2\lesssim \alpha_{\Omega_1,\Omega_2}(x,r).
\end{equation}
This was proved in \cite{FTV1}. The arguments are much more delicate than the ones for \cref{eqalfai2}, and they involve suitable quantitative Faber-Krahn inequalities arising from a remarkable work of Allen, Kriventsov, and Neumayer \cite{AKN}.

\subsection{Dyadic lattices}

Given an $n$-Ahlfors regular set $E\subset\R^{n+1}$ and $\mu=\HH^n|E$, we consider
the dyadic lattice of ``cubes'' introduced by David and Semmes in \cite[Chapter 2]{DS1}. The properties satisfied by $\DD_E$ are the following.
Assume first, for simplicity, that $\diam(E)=\infty$. Then for each $j\in\Z$ there exists a family $\DD_{E,j}$ of Borel subsets of $E$ (the dyadic cubes of the $j$-th generation) such that:
\begin{itemize}
  \item[$(a)$] each $\DD_{E,j}$ is a partition of $E$, i.e.\ $E=\bigcup_{Q\in \DD_{E,j}} Q$ and $Q\cap Q'=\varnothing$ whenever $Q,Q'\in\DD_{E,j}$ and
        $Q\neq Q'$;
  \item[$(b)$] if $Q\in\DD_{E,j}$ and $Q'\in\DD_{E,k}$ with $k\leq j$, then either $Q\subset Q'$ or $Q\cap Q'=\varnothing$;
  \item[$(c)$] for all $j\in\Z$ and $Q\in\DD_{E,j}$, we have $2^{-j}\lesssim\diam(Q)\leq2^{-j}$ and $\mu(Q)\approx 2^{-jn}$;
  \item[$(d)$] there exists $C>0$ such that, for all $j\in\Z$, $Q\in\DD_{E,j}$, and $0<\tau<1$,
        \begin{equation}\label{small boundary condition}
          \begin{split}
            \mu\big(\{x\in Q:\, & \dist(x,E\setminus Q)\leq\tau2^{-j}\}\big) \\&+\mu\big(\{x\in E\setminus Q:\, \dist(x,Q)\leq\tau2^{-j}\}\big)\leq C\tau^{1/C}2^{-jn}.
          \end{split}
        \end{equation}
        This property is usually called the {\em small boundaries condition}.
        From \cref{small boundary condition}, it follows that there is a point $x_Q\in Q$ (the center of $Q$) such that $\dist(x_Q,\supp\mu\setminus Q)\gtrsim 2^{-j}$.
\end{itemize}
We set $\DD_E:=\bigcup_{j\in\Z}\DD_{E,j}$.

In case that $\diam(E)<\infty$, the families $\DD_{E,j}$ are only defined for $j\geq j_0$, with
$2^{-j_0}\approx \diam(E)$, and the same properties above hold for $\DD_E:=\bigcup_{j\geq j_0}\DD_{E,j}$.

Given a cube $Q\in\DD_{E,j}$, we say that its side length is $2^{-j}$, and we denote it by $\ell(Q)$. Notice that $\diam(Q)\leq\ell(Q)$.
We also denote
\begin{equation}\label{defbq}
  B(Q):=B(x_Q,c_1\ell(Q)),\qquad B_Q = B(x_Q,\ell(Q)),
\end{equation}
where $c_1>0$ is some fixed
constant so that $B(Q)\cap\supp\mu\subset Q$, for all $Q\in\DD_E$. Clearly, we have $Q\subset B_Q$.

For $\lambda>1$, we write
$$\lambda Q = \bigl\{x\in \supp\mu:\, \dist(x,Q)\leq (\lambda-1)\,\ell(Q)\bigr\}.$$

The side length of a \enquote{true cube} $P\subset\R^{n+1}$ is also denoted by $\ell(P)$. On the other hand, given a ball $B\subset\R^{n+1}$, its radius is denoted by $r(B)$. For $\lambda>0$, the ball $\lambda B$ is the ball concentric with $B$ with radius $\lambda\,r(B)$.

\subsection{Geometric \texorpdfstring{$\beta$}{β} coefficients}
Given $E\subset\R^{n+1}$, a ball $B$, and a hyperplane $L$, we denote
$$\beta_E(B,L) = \sup_{y\in E\cap B} \frac{\dist(y,L)}{r(B)}$$
and
$$
  b\beta_{E}(B,L) =  \sup_{y\in E\cap B} \frac{\dist(y,L)}{r(B)} +
  \sup_{y\in L\cap B}\!\! \frac{\dist(y,E)}{r(B)} .$$
We set
$$\beta_{E}(B) = \inf_L \beta_{E}(B,L), \qquad b\beta_{E}(B) = \inf_L b\beta_{E}(B,L),$$
where the infimum is taken over all hyperplanes $L\subset\R^{n+1}$. For a ball $B=B(x,r)$, we also write
$$\beta_{E}(x,r)=\beta_{E}(B),\qquad b\beta_{E}(x,r)=b\beta_{E}(B).$$

\subsection{Notation and definitions for graphs}
As we will mostly work with graphs of functions and the
domains they define, we specialize some preceding definitions
to this setting.
We start with \emph{vertical} $\beta$ numbers.
For a function $v\colon\mathbb{R}^n\to \mathbb{R}$,
we define the graph $\beta$-number at $x_0\in\mathbb{R}^n$ and $r>0$ as
\begin{equation}
  \mathfrak b_v(x_0,r)=\frac1r\inf_{a\in\mathrm{Aff}(\mathbb{R}^n)} \lVert v-a \rVert_{L^\infty(B(x_0,r))},
\end{equation}
where $\mathrm{Aff}(\mathbb{R}^n)$ is the space of affine functions on $\mathbb{R}^n$.

For a function $g\colon D\to\R$, where $D$ is either $\R^n$ or a ball therein, we write
\begin{align*}
  \varphi_g(x_0) := (x_0, g(x_0)) \quad\text{and}\quad \Gamma_g := \{ \varphi_g(x_0) : x_0\in D \}.
\end{align*}
We denote the set above the graph $\Gamma_g$ by
\[
  \Omega_g^+ :=\{(x_0,t)\in D\times \R:t>g(x_0)\}.
\]
We also define
\begin{align*}
  \wt\ve_g(x_0,r) := \frac{1}{r^n} \left|  \mathcal{H}^{n}(\Omega_g^+\cap S(\varphi_g(x_0),r)) - \frac{1}{2} \mathcal{H}^{n}(S(\varphi_g(x_0),r)) \right|,
\end{align*}
whenever $B(x_0,r) \subset D$.
Notice that $\wt\ve_g(x_0,r)\leq \ve_{n,\Omega_g^+}(\varphi_g(x_0),r)$.

We say that $g$ satisfies a \emph{Lipschitz condition at $z_0\in D$} if there exists a constant $0\le L<1$ such that
\begin{equation}\label{eq:lipschitz-at-center}
  |g(y_0)-g(z_0)| \le L|y_0-z_0|
  \qquad\text{for every }y_0\in D.
\end{equation}

\section{The difficult implication in the planar case} \label{secplane}
If $f$ is Lipschitz with constant $L<1$, then the graph $\Gamma_f$ and the circle $S(\varphi_f(z),r)$ intersect at exactly two points, and thus $\wt\ve_f(z,r) = \ve_{1,\Omega_f^+}(\varphi_f(z),r)$.
We will use \emph{centered differences} throughout this section:
\[
  \delta_h^2f(z)=f(z+h)+f(z-h)-2f(z).
\]

\begin{lemma}\label{thm:epsilon-to-second-difference}
  Let $f\colon I\to\mathbb{R}$ be $L$-Lipschitz with $L<1$ and $0<\alpha\leq1$. Suppose that
  \[
    \wt\ve_f(z,r)\leq C_{\varepsilon} r^\alpha
  \]
  for some $z\in I$ and $r$ as below. If $h>0$,
  \[
    r :=|\vphi_f(z+h)-\vphi_f(z)| \quad\text{and}\quad [z-r,z+r]\subset I,
  \]
  then
  \begin{equation}
    |\delta_h^2f(z)|\leq M h^{1+\alpha},\qquad
    M=(1+L)(1+L^2)^{(1+\alpha)/2}C_{\varepsilon}.
  \end{equation}
\end{lemma}
\begin{proof}
  By the choice of $r$, the right intersection point of $\Gamma_f\cap S(\vphi_f(z),r)$ is $X_+ :=\vphi_f(z+h)$.
  Write the left one as $X_- :=\vphi_f(z-k)$, where $k\in (0,r]$, and put
  $X :=\vphi_f(z)$ and $v_\pm :=X_\pm-X$.

  Let $\theta\in(0,2\pi)$ be the angle of the arc of $\partial B(\vphi_f(z),r)$ in $\Omega_f^+$. Then
  \[
    \wt{\varepsilon}_f(z,r)=|\theta-\pi|,
  \]
  and elementary circle geometry gives
  \begin{equation}\label{eq:vector-estimate}
    |v_++v_-|=2r|\cos(\theta/2)|
    =2r\sin\bigl({\wt\varepsilon}_f(z,r)/2\bigr)
    \leq C_{\varepsilon} r^{1+\alpha}.
  \end{equation}

  Recall that
  \[
    v_++v_-=(h-k,\ f(z+h)+f(z-k)-2f(z)).
  \]
  Consequently,
  \begin{align*}
    |\delta_h^2f(z)|
     & \leq |f(z+h)+f(z-k)-2f(z)|+|f(z-h)-f(z-k)| \\
     & \leq |v_++v_-|+L|h-k|                      \\
     & \leq(1+L)|v_++v_-|                         \\
     & \leq(1+L)C_{\varepsilon} r^{1+\alpha}.
  \end{align*}
  Finally,
  \[
    h\leq r=\sqrt{h^2+|f(z+h)-f(z)|^2}\leq\sqrt{1+L^2}\,h,
  \]
  which proves the claim.
\end{proof}

In the following we prove that decay of second differences on an interval $I\subset \mathbb{R}$ implies $C^{1,\alpha}$-regularity on that interval.
Alternatively, and this is the approach we will choose in higher dimensions, it is possible to multiply by a bump function so that one can work with functions defined on all of $\mathbb{R}$.
In the planar case it is not difficult to show that the product retains the decay of the second order differences by expanding the second difference of the product.

\begin{lemma}\label{lem:second-difference}
  Let $f\colon I\to\mathbb{R}$ be continuous and $0<\alpha\leq1$. Suppose
  \begin{equation}\label{eq:criterion-hypothesis}
    |\delta_h^2f(z)|\leq M|h|^{1+\alpha}
    \quad\text{whenever }z-h,z+h\in I.
  \end{equation}
  Then $f\in C^1(I)$ and
  \begin{equation}\label{eq:holder-constant}
    |f'(x)-f'(y)|\leq\frac{M}{2^\alpha-1}|x-y|^\alpha
    \qquad\forall x,y\in I.
  \end{equation}
\end{lemma}

\begin{proof}
  For $t>0$ with $[x,x+t]\subset I$, set
  \[
    q_t^+(x)=\frac{f(x+t)-f(x)}t.
  \]
  The hypothesis, applied at the midpoint of $[x,x+t]$, gives
  \begin{equation}\label{eq:quotient-step}
    |q_t^+(x)-q_{t/2}^+(x)|
    =\frac{|\delta_{t/2}^2f(x+t/2)|}{t}
    \leq\frac{M}{2^{1+\alpha}}t^\alpha.
  \end{equation}
  Fix a compact subinterval $J\subset I$ and choose $t_0>0$ small enough such that $q_{t_0}^+$ is defined on $J$.
  Summing \cref{eq:quotient-step} shows
  that $(q_{2^{-m}t_0}^+)_m$ converges uniformly to a continuous function $v$.

  We now show that $v$ is indeed the derivative.
  For any $a<b$ in the interior of $J$ and $t_m=2^{-m}t_0$,
  \begin{align*}
    \int_a^b q_{t_m}^+(x)\,dx
    =\frac1{t_m}\bigg( \int_b^{b+t_m}f(s)\,\mathrm{d} s - \int_a^{a+t_m}f(s)\,\mathrm{d} s \bigg)
    \longrightarrow f(b)-f(a)
  \end{align*}
  by continuity of $f$. Uniform convergence gives
  $f(b)-f(a)=\int_a^b v(x)\,dx$. Since $v$ is continuous, the fundamental
  theorem of calculus implies $f'=v$. Applying this on compact subintervals
  proves $f\in C^1(I)$.

  Now telescope \cref{eq:quotient-step} starting from any admissible $t$.
  As the quotients converge to $f'(x)$, we obtain
  \begin{equation}
    |f'(x)-q_t^+(x)|
    \leq\frac{M}{2^{1+\alpha}}t^\alpha\sum_{m=0}^\infty2^{-m\alpha}
    =\frac{M}{2(2^\alpha-1)}t^\alpha.
  \end{equation}
  The same argument for backward quotients
  $q_t^-(y)=(f(y)-f(y-t))/t$ gives the identical bound. If $x<y$ and $t=y-x$,
  then
  \[
    q_t^+(x)=\frac{f(y)-f(x)}{y-x}=q_t^-(y).
  \]
  Adding the two bounds proves \cref{eq:holder-constant}.
\end{proof}

\begin{theorem}\label{thm:planar-interval}
  Let $0<\alpha\leq1$ and let $f\colon (-1,1)\to\mathbb{R}$ be $L$-Lipschitz, with
  $L<1$. Suppose
  \[
    \wt\ve_f(z,r)\leq C_{\varepsilon} r^\alpha
    \quad\text{for all }z\in(-1,1)\text{ and }r>0\text{ with }[z-r,z+r]\subset(-1,1).
  \]
  Then $f\in C^{1,\alpha}((-1/4,1/4))$, and
  \begin{equation}\label{eq:local-final-bound}
    [f']_{C^{0,\alpha}((-1/4,1/4))}
    \leq \frac{(1+L)(1+L^2)^{(1+\alpha)/2}}{2^\alpha-1}C_{\varepsilon}.
  \end{equation}
\end{theorem}

\begin{proof}
  Let $J=(-1/4,1/4)$. If $z-h,z+h\in J$ and $h>0$, then $|z|<1/4$ and
  $h<1/4$. Hence,
  \[
    r=|\vphi_f(z+h)-\vphi_f(z)|\leq\sqrt{1+L^2}\,h<\frac{\sqrt2}{4}.
  \]
  In particular, $[z-r,z+r]\subset(-1,1)$.
  \Cref{thm:epsilon-to-second-difference} therefore gives \cref{eq:criterion-hypothesis} on $J$.
  \Cref{lem:second-difference} now proves the claim.
\end{proof}
\vv

\section{The case of a Lipschitz graph with small constant}\label{secLipschitgraph}

\begin{theorem}\label{teolipschitz}
  Let $0<\alpha<1$, let $\Gamma\subset\R^{n+1}$ be the Lipschitz graph through the origin of a function $f:\R^n\to\R$ with slope $L\leq \lambda_0$,  for
  some fixed $0<\lambda_0<1/2$ small enough, and let $\Omega$ be one of the connected components of
  $\R^{n+1}\setminus \Gamma$. Suppose that
  $$\ve_n(x,r) \leq C_\ve\,r^\alpha\quad \mbox{ for all $x\in\Gamma\cap B(0,1)$, $0<r\leq1$.}$$
  Then $f\in C^{1,\alpha}(B(0,1/4))$, with
  $$\|\nabla f\|_{\dot C^\alpha(B(0,1/4))} \lesssim C(\alpha,\lambda_0)\,C_\ve.$$
\end{theorem}

The strategy of the proof is to relate the geometric coefficients to the Littlewood–Paley decomposition of the function.
In \cref{sec:calculations-on-Rn}, we first consider functions defined on $\R^n$ and establish the basic comparison between these quantities.
In \cref{sec:restricted-g}, we introduce the localized function and extend it to $\R^n$ by multiplying it by a smooth cutoff.
\Cref{sec:LP-bound-by-beta} then establishes the decay of its Littlewood–Paley pieces up to an error involving the graph $\beta$-numbers, multiplied by the square of the Lipschitz constant.
Conversely, \cref{sec Littlewood-Paley} controls the graph $\beta$-numbers in terms of the Littlewood–Paley pieces.
Finally, in \cref{sec:absorption}, we combine these estimates and use the smallness of the Lipschitz constant to absorb the error and obtain the decay required for \(C^{1,\alpha}\) regularity.

\subsection{Calculations for functions on \texorpdfstring{$\mathbb{R}^n$}{ℝⁿ}}\label{sec:calculations-on-Rn}
We will focus on the case where $\partial \Omega^+$ is given by the graph of a function.
Let $g\colon\mathbb{R}^n\to\mathbb{R}$ and suppose that it satisfies a Lipschitz condition at a point with constant $L$.
The estimates below will work pointwise at that fixed center and with varying radii $r>0$.

We define the Littlewood--Paley operators as in \cite[Section 1.4.2]{grafakosModernFourierAnalysis2014}.
Let $\xi\in\mathcal{S}(\mathbb{R}^n)$ be a radial Schwartz function with
$\mathbf{1}_{A(0,1,2-\frac{2}{7})} \le \widehat \xi \le \mathbf{1}_{A(0, 1-\frac{1}{7}, 2)}$ and
\begin{align*}
  \sum_{j\in\mathbb{Z}} \widehat{\xi}(2^{-j} \zeta) = 1,
\end{align*}
for all $\zeta\neq 0$.
We define
$P_j f := \xi_{2^{-j}} \ast f$, where $\xi_t(x_0) := \frac{1}{t^n} \xi(\frac{x_0}{t})$.
We also define the Schwartz function $\varphi$ by setting
\begin{align*}
  \widehat \varphi(\zeta)
  = \begin{cases}
      \sum_{j\le 0} \widehat{\xi}(2^{-j}\zeta) & \text{if } \zeta\neq 0, \\
      1                                        & \text{if } \zeta=0.
    \end{cases}
\end{align*}
Additionally, $S_j f := \varphi_{2^{-j}} \ast f$.
Note that
\begin{align*}
  \int \varphi \,\mathrm{d}x_0 = 1, \quad
  P_j = S_j - S_{j-1} \quad\text{and}\quad
  \int (x_0)^{\nu} \xi(x_0) \,\mathrm{d}x_0 =0,
\end{align*}
for every multi-index $\nu$.

Furthermore, let $\psi\colon\mathbb{R}\to\mathbb{R}$ be the even radial profile of $\xi$, i.e. $\psi(t) := \xi(t e_1)$.
Define $\eta,\rho\colon\mathbb{R}^{n+1}\to\mathbb{R}$,
\begin{align*}
  \eta(x) & := \psi(|x_0|) \mathbf{1}_{K}(x), \\
  \rho(x) & := \psi(|x|),
\end{align*}
where $K$ is a double-sided cone with aperture approximately $L$, i.e.
\begin{align*}
  K := \{ x\in\mathbb{R}^{n+1} : |x_{n+1}| \le C_K L |x_0| \},
\end{align*}
where $C_K>1$ will be chosen later.

We define integrated versions of the Carleson coefficient to be
\begin{align*}
  a_{\eta,\Omega^+}(x,r) & := \frac{1}{r^{n+1}} \left|  \int_{\Omega^+} \eta\Bigl(\frac{y-x}{r}\Bigr) \,\mathrm{d}\mathcal{H}^{n+1}(y) -  \int_{H_x^+} \eta\Bigl(\frac{y-x}{r}\Bigr) \,\mathrm{d}\mathcal{H}^{n+1}(y) \right|, \\
  a_{\rho,\Omega^+}(x,r) & := \frac{1}{r^{n+1}} \left|  \int_{\Omega^+} \rho\Bigl(\frac{y-x}{r}\Bigr) \,\mathrm{d}\mathcal{H}^{n+1}(y) -  \int_{H_x^+} \rho\Bigl(\frac{y-x}{r}\Bigr) \,\mathrm{d}\mathcal{H}^{n+1}(y) \right|,
\end{align*}
where $H_x^+$ is a half-space associated to some arbitrary hyperplane $H_x$ passing through $x$.
The definition is independent of the choice of $H_x$ since $\eta$ and $\rho$ are integrable and even.
We will only choose hyperplanes with slope bounded by $C_KL$, so that they are contained in the cone.

\begin{lemma}\label{thm:a-eta-decay}
  Let $x_0\in\R^n$ and $x=\varphi_g(x_0)$.
  If $\wt\ve_g(x_0,t) \lesssim t^\alpha$ for all $t>0$, then
  \begin{align*}
    a_{\rho,\Omega_g^+}(x,r) \lesssim cr^\alpha.
  \end{align*}
\end{lemma}

\begin{proof}
  Since $\rho$ is radial, we have
  \begin{align*}
    a_{\rho,\Omega_g^+}(x,r)
     & = \frac{1}{r^{n+1}} \left|  \int_0^{\infty} \psi\Big(\frac{t}{r}\Big) \mathcal{H}^{n}(\Omega_g^+\cap S(x,t)) \,\mathrm{d} t -  \int_0^{\infty} \psi\Big(\frac{t}{r}\Big) \mathcal{H}^{n}(H_x^+\cap S(x,t)) \,\mathrm{d} t  \right| \\
     & \le \frac{1}{r^{n+1}} \int_0^{\infty} \left|  \psi\Big(\frac{t}{r}\Big) \right| \left| \mathcal{H}^{n}(\Omega_g^+\cap S(x,t)) - \frac{1}{2}\mathcal{H}^{n}(S(x,t)) \right| \,\mathrm{d} t                                          \\
     & \leq \frac{1}{r^{n+1}} \int_0^{\infty} \left|  \psi\Big(\frac{t}{r}\Big) \right| \wt\varepsilon_g(x_0,t) t^n \,\mathrm{d} t                                                                                                        \\
     & \lesssim \frac{1}{r^{n+1}} \int_0^{\infty} \left|  \psi\Big(\frac{t}{r}\Big) \right| t^{n+\alpha} \,\mathrm{d} t
    = r^\alpha \int_{0}^{\infty} \left| \psi(s) \right| s^{n+\alpha} \,\mathrm{d} s                                                                                                                = cr^{\alpha}.
  \end{align*}
\end{proof}

\begin{lemma}\label{thm:LP-equals-a-rho}
  Let $g\colon\mathbb{R}^n\to\mathbb{R}$ satisfy a Lipschitz condition at $x_0$.
  Then,
  \begin{align*}
    |\xi_r \ast g (x_0)| = r \,a_{\eta,\Omega_g^+}(\varphi_g(x_0), r).
  \end{align*}
\end{lemma}
\begin{proof}
  Set $x=\varphi_g(x_0)$ and $K_x=x+K$. By the
  independence of the half-space in the definition of $a_{\eta,\Omega_g^+}$, choose $H_x=\{y:y_{n+1}=g(x_0)\}$ and
  its upper half-space.
  We calculate
  \begin{align*}
    r^{n+1} a_{\eta,\Omega_g^+}(x,r)
     & = \left|  \int_{\Omega_g^+} \eta\Big(\frac{y-x}{r}\Big) \,\mathrm{d}\mathcal{H}^{n+1}(y) -  \int_{H_x^+} \eta\Big(\frac{y-x}{r}\Big) \,\mathrm{d}\mathcal{H}^{n+1}(y) \right|                                                     \\
     & = \left|  \int_{\mathbb{R}^n} \psi\bigg(\frac{|y_0-x_0|}{r}\bigg) \left( \int_{g(y_0)}^{F_{K_x}(y_0)} \,\mathrm{d}y_{n+1}  - \int_{g(x_0)}^{F_{K_x}(y_0)} \,\mathrm{d}y_{n+1} \right) \,\mathrm{d}\mathcal{H}^{n}(y_0)  \right|   \\
     & =  \left|  \int_{\mathbb{R}^n} \xi\bigg(\frac{y_0-x_0}{r}\bigg) ( g(x_0) - g(y_0)) \,\mathrm{d}\mathcal{H}^{n}(y_0)  \right|                                                                                                      \\
     & =  \left|  \int_{\mathbb{R}^n} \xi\bigg(\frac{y_0-x_0}{r}\bigg) g(y_0) \,\mathrm{d}\mathcal{H}^{n}(y_0)  \right|                                                                                        = r^{n} |\xi_r * g(x_0)|,
  \end{align*}
  where $F_{K_x}$ denotes the function corresponding to the cone $x+K$.
  We used that $\int \xi \,\mathrm{d}x_0 = 0$.
\end{proof}

\begin{lemma}\label{thm:difference-of-a-coeff}
  Let $g\colon\mathbb{R}^n\to\mathbb{R}$ satisfy a Lipschitz condition at $x_0$.
  For any $N\in\mathbb{N}$ large enough, we have
  \begin{align*}
    \left|  a_{\eta,\Omega_g^+}(\varphi_g(x_0),r) - a_{\rho,\Omega_g^+}(\varphi_g(x_0),r)  \right|
    \lesssim_N L^2 \sum_{j\ge 1} \mathfrak b_g(x_0,2^jr) 2^{-Nj}.
  \end{align*}
\end{lemma}

\begin{proof}
  We first rewrite $a_\rho(x,r)$:
  Let $K_x^+ := \{ y\in\mathbb{R}^{n+1} : y_{n+1} - x_{n+1} > C_K L |y_0 - x_0| \}$ be the {\em upper opening} of the cone $K_x$.
  If we choose $H_x$ to be contained in the cone $x+K$, we have $K_x^+ \subset \Omega_g^+ \cap H_x^+$, and thus for $x\in \partial \Omega_g^+$ and $r>0$,
  \begin{align*}
    a_{\rho,\Omega_g^+}(x,r) &
    = \frac{1}{r^{n+1}} \bigg|  \left( \int_{\Omega_g^+} \rho\Big(\frac{y-x}{r}\Big) \,\mathrm{d} y - \!\int_{K_x^+} \rho\Big(\frac{y-x}{r}\Big) \,\mathrm{d} y \right) \\
                             & \qquad \quad
    - \left( \int_{H_x^+} \rho\Big(\frac{y-x}{r}\Big) \,\mathrm{d} y  -\! \int_{K_x^+} \rho\Big(\frac{y-x}{r}\Big) \,\mathrm{d} y \right)
    \bigg|                                                                                                                                                              \\
                             & = \frac{1}{r^{n+1}} \left|  \int_{\Omega_g^+\cap K_x} \rho(\frac{y-x}{r}) \,\mathrm{d} y
    - \int_{H_x^+\cap K_x} \rho(\frac{y-x}{r}) \,\mathrm{d} y
    \right|.
  \end{align*}
  We use this to calculate
  \begin{align*}
     & \left|  a_\eta(x,r) - a_\rho(x,r)  \right|                                                                          \\
     & \le \frac{1}{r^{n+1}} \bigg|
    \left( \int_{\Omega_g^+\cap K_x} \psi\bigg(\frac{|y_0-x_0|}{r}\bigg) \,\mathrm{d} y
    - \int_{H_x^+\cap K_x} \psi\bigg(\frac{|y_0-x_0|}{r}\bigg) \,\mathrm{d} y
    \right)                                                                                                                \\
     & \qquad\qquad\qquad -
    \left( \int_{\Omega_g^+\cap K_x} \psi\bigg(\frac{|y-x|}{r}\bigg) \,\mathrm{d} y
    - \int_{H_x^+\cap K_x} \psi\bigg(\frac{|y-x|}{r}\bigg) \,\mathrm{d} y
    \right)
    \bigg|                                                                                                                 \\
     & \le \frac{1}{r^{n+1}} \int_{(\Omega_g^+ \Delta H_x^+) \cap K_x}
    \left|  \psi\bigg(\frac{|y_0-x_0|}{r}\bigg) - \psi\bigg(\frac{|y-x|}{r}\bigg) \right| \,\mathrm{d}\mathcal{H}^{n+1}(y) \\
     & \le \frac{1}{r^{n+1}} \sum_{k\ge 0}
    \int_{A_{k,x,r}\cap (\Omega_g^+ \Delta H_x^+) \cap K_x} \lVert \psi' \rVert_{L^{\infty}((2^k - 1, c2^{k+1}))}
    \left|  \frac{|y_0-x_0| - |y-x| }{r}  \right|  \,\mathrm{d}\mathcal{H}^{n+1}(y)                                        \\
  \end{align*}
  where $A_{0,x,r} := B(x_0, 2r)\times \mathbb{R}$ and
  \begin{align*}
    A_{k,x,r} := \{ y \in \mathbb{R}^{n+1} : |y_0 - x_0| \in (2^kr, 2^{k+1}r) \} ,\qquad\forall k\ge 1.
  \end{align*}
  When applying the mean value theorem, note that $y\in A_{k,x,r}$ implies that $|y_0-x_0|/r \in (2^k - 1,2^{k+1})$.
  Since $y\in K_x$, we also have $|y-x|/r \lesssim |y_0-x_0|/r \lesssim 2^{k+1}$,
  and
  \begin{align*}
    \left|  |y_0-x_0| - |y-x|  \right| \le |1-\cos(\arctan C_K L)| |y-x| \lesssim 2^k L^2 r.
  \end{align*}
  Since $\psi$ is a Schwartz function, we have $\lVert \psi' \rVert_{L^{\infty}((2^k - 1, c2^{k+1}))} \le c(M) 2^{-Mk}$ for any $M\in \mathbb{N}$.

  We use affine approximating functions on consecutive dyadic scales to estimate the Hausdorff measure.
  For $j\ge 1$, let $\ell_j$ denote the affine function minimizing $\mathfrak b_g(x_0, 2^jr)$.
  Let $\ell_0 := \ell_1 - \ell_1(x_0) +g(x_0)$ and $H_x$ be the hyperplane associated to it.
  The Lipschitz condition of $g$ at $x_0$ implies that $|\nabla \ell_1| \le 2L$, and indeed $H_x$ is contained in the cone $x+K$:
  \begin{align*}
    2r|\nabla \ell_1|
    \le \lVert \ell_1 - g(x_0) \rVert_{L^{\infty}(B(x_0,2r))}
    \le \lVert \ell_1 - g \rVert_{L^{\infty}(B(x_0,2r))} + \lVert g - g(x_0) \rVert_{L^{\infty}(B(x_0,2r))}
    \le 4Lr.
  \end{align*}
  We have $\lVert \ell_1 - \ell_0 \rVert_{L^{\infty}(B(x_0,2r))} = |\ell_1(x_0) - g(x_0)| \le 2r \mathfrak b_g(x_0,2r)$.
  For $k+1\ge j\ge 2$,
  \begin{align*}
    \lVert \ell_j - \ell_{j-1} \rVert_{L^{\infty}(B(x_0,2^{k+1}r))}
     & \le \frac{2^{k+1}r}{2^{j-1}r}\lVert \ell_j - \ell_{j-1} \rVert_{L^{\infty}(B(x_0,2^{j-1}r))} \\
     & \le 2^{k+1}r ( \mathfrak b_g(x_0,2^{j-1}r) + 2 \mathfrak b_g(x_0,2^{j}r)).
  \end{align*}
  We have
  \begin{align*}
    \mathcal{H}^{n+1}(A_{k,x,r}\cap (\Omega_g^+ \Delta H_x^+))
     & \le \int_{B(x_0, 2^{k+1}r)} |g - \ell_0| \,\mathrm{d}\mathcal{H}^{n}                                                       \\
     & \lesssim (2^{k+1}r)^n \Big\| g - \ell_{k+1} + \sum_{j=1}^{k+1}  (\ell_j - \ell_{j-1}) \Big\|_{L^{\infty}(B(x_0,2^{k+1}r))} \\
     & \lesssim (2^{k+1}r)^{n+1} \sum_{j=1}^{k+1} \mathfrak b_g(x_0,2^jr).
  \end{align*}

  Combining all estimates yields
  \begin{align*}
    \left|  a_{\eta,\Omega_g^+}(x,r) - a_{\rho,\Omega_g^+}(x,r)  \right|
     & \lesssim \frac{L^2}{r^{n+1}} \sum_{k\ge 0} 2^{(1-M)k} \sum_{j=1}^{k+1} \mathfrak b_g(x_0,2^jr) (2^kr)^{n+1} \\
     & \lesssim L^2 \sum_{j\ge 1} \mathfrak b_g(x_0,2^jr) \sum_{k\ge j-1} 2^{(n+2-M)k}                             \\
     & \lesssim L^2 \sum_{j\ge 1} \mathfrak b_g(x_0,2^jr) 2^{-(M-n-2)j},
  \end{align*}
  for sufficiently large $M$.
\end{proof}

Combining \cref{thm:LP-equals-a-rho,thm:difference-of-a-coeff} yields
\begin{lemma}\label{thm:LP-beta-general}
  Let $g\colon\mathbb{R}^n\to\mathbb{R}$ satisfy a Lipschitz condition at $x_0$.
  Then, for any sufficiently large $N\in\mathbb N$,
  \begin{align*}
    |\xi_r \ast g(x_0)|
     & \lesssim ra_{\rho,\Omega_g^+}(\varphi_g(x_0),r) + L^2 r \sum_{h\ge 1} \mathfrak b_g(x_0,2^h r) 2^{-Nh}.
  \end{align*}
\end{lemma}

If we also assume decay of the $\ve$-coefficients,
\cref{thm:a-eta-decay} yields the following.
\begin{lemma}\label{thm:LP-beta-estimate}
  Let $g\colon\mathbb{R}^n\to\mathbb{R}$ satisfy a Lipschitz condition at $x_0$,
  and let $\wt\ve_g(x_0,t) \lesssim t^\alpha$ for all $t>0$.
  Then, for any sufficiently large $N\in\mathbb N$,
  \begin{align*}
    |\xi_r \ast g(x_0)|
     & \lesssim r^{1+\alpha} + L^2 r \sum_{h\ge 1} \mathfrak b_g(x_0,2^h r) 2^{-Nh}.
  \end{align*}
\end{lemma}

\begin{rem}\label{thm:lipschitz-at-x0-suffices}
  It suffices to require that there exists $r_0>0$ such that $\wt\ve_g(x_0,r)\le C_\varepsilon r^\alpha$ for $r\in (0,r_0)$.
  Since $\wt\ve_g(x_0,\cdot)$ is uniformly bounded, we have
  \begin{align*}
    \wt\ve_g(x_0,r)\lesssim (C_\varepsilon + r_0^{-\alpha}) r^\alpha, \quad\forall r>0.
  \end{align*}
\end{rem}
\vv

\subsection{Small Lipschitz constant on \texorpdfstring{$B(0,1)$}{B(0,1)}}\label{sec:restricted-g}
Consider an $L$-Lipschitz function
$f\colon B(0,1)\to\mathbb{R}$ with $f(0)=0$ and
$L<\lambda_0$, where $\lambda_0>0$ will be chosen later.
Suppose
\[
  \wt\ve_f(x_0,r)\leq C_{\varepsilon} r^\alpha
  \quad\text{for all }x_0\in B(0,1)\text{ and }r>0\text{ with }B(x_0,r)\subset
  B(0,1).
\]
Let $F$ be the zero extension of $f$ and consider $u=\chi F$, where $\chi\in C_c^\infty(\mathbb{R}^n)$ with $\mathbf{1}_{B(0,1/4)} \le \chi \le \mathbf{1}_{B(0,1/2)}$.
If not specified otherwise, $f,u$ and $F$ will be as above for the rest of the section.
The function $u$ is compactly supported and globally Lipschitz:
We first claim that $F$ satisfies a Lipschitz condition at $x_0\in\supp \chi$, i.e.
\begin{align}\label{eq:f-lipschitz-condition}
  |F(x_0) - F(y_0)| \le 2L |x_0 - y_0|, \qquad \forall y_0\in\mathbb{R}^n.
\end{align}
For $y_0 \in B(0,1)$ this follows from the assumption on $f$.
For $y_0\not\in B(0,1)$, we have $|x_0 - y_0| \ge \frac{1}{2}$, and thus
\begin{align*}
  |F(y_0) - F(x_0)|
  = |F(x_0)|
  \le \lVert f \rVert_{\infty}
  \le 2 L |x_0 - y_0|.
\end{align*}
This also implies that $u$ is Lipschitz:
Without loss of generality, let $x_0\in\supp \chi$ and $ y_0\in\mathbb{R}^n$.
We have
\begin{align}\label{eq:fg-lipschitz}
  |F(x_0)\chi(x_0) - F(y_0)\chi(y_0)|
   & \le |F(x_0)| |\chi(x_0) - \chi(y_0)| + |\chi(y_0)| |F(x_0) - F(y_0)| \nonumber \\
   & \le L \lVert \nabla \chi \rVert_{\infty} |x_0-y_0|
  + 2L\lVert \chi \rVert_{\infty} |x_0 - y_0|.
\end{align}

Note that $P_j u, \,S_k u\in\mathcal{S}(\mathbb{R}^n)$ since $u$ is integrable and compactly supported.

\vv

\subsection{Littlewood--Paley estimates in terms of graph
  \texorpdfstring{$\beta$}{β}-numbers}\label{sec:LP-bound-by-beta}

\begin{lemma}\label{thm:beta-graph-product}
  Let $f,g\colon\mathbb{R}^n\to \mathbb{R}$.
  We have
  \begin{align*}
    |g(x_0)| \mathfrak b_f(x_0,r)
    \le \mathfrak b_{fg}(x_0,r) & +  |f(x_0)|  \mathfrak b_g(x_0,r)
    +  \sup_{y_0\in B(x_0,r)} \frac{|f(x_0) - f(y_0)| |g(x_0) - g(y_0)|}r.
  \end{align*}
\end{lemma}
\begin{proof}
  If $g(x_0) \neq 0$, consider the affine function $A\colon\mathbb{R}^n\to\mathbb{R}$ defined by
  \begin{align*}
    A(y_0) := \frac{L_{fg}(y_0) - f(x_0) L_g(y_0) + f(x_0) g(x_0)}{g(x_0)},
  \end{align*}
  where $L_{fg}$ and $L_g$ are the minimizers of $\mathfrak b_{fg}(x_0,r)$ and $\mathfrak b_{g}(x_0,r)$, respectively.
  We calculate
  \begin{align*}
    g(x_0) (f(y_0) - A(y_0))
    = ( f(y_0)g(y_0) - L_{fg}(y_0) )
     & + f(x_0) ( L_g(y_0)-g(y_0) )          \\
     & - (f(x_0) - f(y_0))(g(x_0) - g(y_0)).
  \end{align*}
  Dividing by $r$, taking absolute values and $\sup_{y_0\in B(x_0,r)}$ yields the claim.
\end{proof}

We also have:

\begin{lemma}\label{thm:beta-graph-product*}
  Let $f,g\colon\mathbb{R}^n\to \mathbb{R}$.
  We have
  \begin{align*}
    \mathfrak b_{fg}(x_0,r) \le  |g(x_0)| \mathfrak b_f(x_0,r) & + |f(x_0)| \mathfrak b_g(x_0,r)
    + \sup_{y_0\in B(x_0,r)} \frac{|f(x_0) - f(y_0)| |g(x_0) - g(y_0)|}r.
  \end{align*}
\end{lemma}

\begin{proof}
  Consider the affine function $A\colon\mathbb{R}^n\to\mathbb{R}$ defined by
  \begin{align*}
    A(y_0) = g(x_0) L_f(y_0)  + f(x_0) L_g(y_0) - f(x_0) g(x_0),
  \end{align*}
  where $L_f$ and $L_g$ are the minimizers of $\mathfrak b_{f}(x_0,r)$ and $\mathfrak b_{g}(x_0,r)$, respectively.
  Then we have
  \begin{align*}
    f(y_0)g(y_0) - A(y_0) =
    g(x_0) (f(y_0)- L_f(y_0)) & + f(x_0) (g(y_0) - L_g(y_0))          \\
                              & + (f(y_0) - f(x_0))(g(y_0) - g(x_0)).
  \end{align*}
  Dividing by $r$, taking absolute values and $\sup_{y_0\in B(x_0,r)}$ yields the lemma.
\end{proof}

\begin{lemma}\label{thm:Aj-upper-bound-bk}
  Assume that $\wt\ve_f(x_0,r)\lesssim r^\alpha$.
  We have
  \begin{align*}
    \left|  P_j u(x_0)  \right| \lesssim 2^{-(1+\alpha)j} + L^2 2^{-j} \sum_{h\ge 1} 2^{-Nh} \mathfrak{b}_u(x_0, 2^{h-j}), \qquad j\in\mathbb{Z}, \, x_0\in\mathbb{R}^n.
  \end{align*}
\end{lemma}

\begin{proof}
  Since $\lVert P_j u  \rVert_{\infty} \lesssim 2^{-j}\lVert u \rVert_{\mathrm{Lip}}$, the estimate follows for $j<0$.

  Now let $j\ge 0$.
  A Taylor expansion of $\chi$ yields
  \begin{align}\label{eqbet**1}
    \left|  P_j u(x_0)  \right|
     & = \left|  \int \xi_{2^{-j}}(x_0 -y_0) \chi(y_0) F(y_0) \,\mathrm{d}y_0  \right|                                                  \\
     & \lesssim  \left|  \int \xi_{2^{-j}}(x_0-y_0)\chi(x_0) F(y_0) \,\mathrm{d}y_0  \right|   \notag                                   \\
     & \quad\quad +
    \left|  \sum_{|\beta| = 1}\int \xi_{2^{-j}}(x_0-y_0) \partial^\beta \chi(x_0)(y_0-x_0)^\beta F(y_0) \,\mathrm{d}y_0  \right| \notag \\
     & \quad\quad + \lVert \nabla^2 \chi \rVert_{\infty} \int |\xi_{2^{-j}}(x_0-y_0)| |y_0-x_0|^2 |F(y_0)| \,\mathrm{d}y_0.\notag
  \end{align}

  For the remainder term we have
  \begin{align*}
    \lVert \nabla^2 \chi \rVert_{\infty} \int |\xi_{2^{-j}}(x_0-y_0)| |x_0-y_0|^2 |F(y_0)| \,\mathrm{d}y_0
     & \lesssim \lVert \nabla^2 \chi \rVert_{\infty} \lVert F \rVert_{\infty} \int |\xi(z_0)||z_0|^2 2^{-2j} \,\mathrm{d} z_0 \\
     & \lesssim 2^{-2j}.
  \end{align*}

  To deal with the term with $|\beta|=1$, using vanishing moments, we write
  \begin{align*}
    \bigg|  \int \xi_{2^{-j}}(x_0-y_0) & \partial^\beta \chi(x_0)(y_0-x_0)^\beta F(y_0) \,\mathrm{d}y_0  \bigg|                                                     \\
                                       & = |\partial^\beta \chi(x_0)| \left|  \int \xi_{2^{-j}}(x_0-y_0) (y_0-x_0)^\beta (F(y_0) - F(x_0)) \,\mathrm{d}y_0  \right| \\
                                       & \lesssim \lVert \nabla \chi \rVert_{\infty} L \int |\xi_{2^{-j}}(x_0-y_0)| |y_0-x_0|^2\,\mathrm{d}y_0
    \lesssim 2^{-2j},
  \end{align*}
  where we can use \cref{eq:f-lipschitz-condition} since $x_0\in\supp \chi$.

  For the term with $\beta=0$:
  Following \cref{thm:lipschitz-at-x0-suffices}, we can apply \cref{thm:LP-beta-estimate} since $F$ satisfies a Lipschitz condition at $x_0\in\supp \chi$.
  Subsequently applying \cref{thm:beta-graph-product} yields
  \begin{align*}
    \chi(x_0) \left|  P_j F(x_0)  \right|
     & \lesssim 2^{-(1+\alpha)j} + L^2 2^{-j} \chi(x_0) \sum_{h\ge 1} 2^{-Nh} \mathfrak{b}_F(x_0, 2^{h-j}) \\
     & \lesssim 2^{-(1+\alpha)j} + L^2 2^{-j} (
    \sum_{h\ge 1} 2^{-Nh} \mathfrak{b}_u(x_0, 2^{h-j}) + \lVert f \rVert_{\infty} 2^{-j} + 2^{-j}
    )                                                                                                      \\
     & \lesssim 2^{-(1+\alpha)j} + L^2 2^{-j} \sum_{h\ge 1} 2^{-Nh} \mathfrak{b}_u(x_0, 2^{h-j}).
  \end{align*}

  Combining all estimates yields the claim.
\end{proof}

\subsection{Affine approximations and Littlewood--Paley pieces}\label{sec Littlewood-Paley}
Define, for $j,k\in\mathbb Z$
\begin{align*}
  A_j:=2^j\lVert P_j u\rVert_{L^\infty(\mathbb R^n)} \quad\text{and}\quad
  \mathfrak b_k := \sup_{x_0\in\mathbb{R}^n} \mathfrak b_u (x_0, 2^{-k}).
\end{align*}

\begin{lemma}
  Let $u\in C_c(\mathbb{R}^n)$ be Lipschitz.
  For any $k\in\mathbb{Z}$, we have
  \begin{align}
    u= S_k u + \sum_{j>k} P_j u \quad\text{and}\quad
    \lVert \nabla^2 S_k u \rVert_{\infty} \lesssim \sum_{j\le k} 2^{2j} \lVert P_j u \rVert_{\infty},
  \end{align}
  with absolute convergence in the uniform norm.
\end{lemma}

\begin{proof}
  First, we have
  \begin{align}
    \left|  S_M u (z) - u(z)  \right|
    \le \int |\varphi_{2^{-M}}(w)| |u(z-w) - u(z) | \,\mathrm{d} w
    \le \lVert u \rVert_{\mathrm{Lip}} \int |\varphi_{2^{-M}}(w)| |w| \,\mathrm{d} w
    \lesssim 2^{-M},
  \end{align}
  so $\lVert S_M u - u \rVert_{\infty} \to 0$ as $M\to \infty$.
  Similarly,
  \begin{align*}
    \lVert P_j u  \rVert_{\infty} \le \lVert u \rVert_{\mathrm{Lip}} \int |\xi_{2^{-j}}(w)||w| \,\mathrm{d} w \lesssim 2^{-j}.
  \end{align*}
  Together, this implies the first claim.

  Next, we have
  \begin{align*}
    \nabla^2 S_k u = \nabla^2 S_M u + \sum_{M<j\le k} \nabla^2 P_j u, \quad M<k.
  \end{align*}
  Young's convolution inequality implies
  \begin{align*}
    \lVert  \nabla^2 S_M u  \rVert_{\infty}
    \le \lVert  \nabla^2 \varphi_{2^{-M}}  \rVert_{1} \lVert u \rVert_{\infty}
    \le 2^{2M} \lVert \nabla^2 \varphi \rVert_1 \lVert u \rVert_{\infty}
    \to 0,
  \end{align*}
  as $M\to -\infty$.

  Now \cite[Corollary 2.2.4]{grafakosModernFourierAnalysis2014} implies
  \begin{align*}
    \lVert \nabla^2 P_j u \rVert_{\infty} \lesssim 2^{2j} \lVert P_j u \rVert_{\infty}.
  \end{align*}
  Therefore,
  \begin{align*}
    \lVert \nabla^2 S_k u \rVert_{\infty} \lesssim \sum_{j\le k} 2^{2j} \lVert P_j u \rVert_{\infty}.
  \end{align*}
\end{proof}

\begin{lemma}\label{thm:bk-upper-bound-Aj}
  Let $u\in C_c(\mathbb{R}^n)$ be Lipschitz.
  We have
  \begin{align*}
    \mathfrak b_k \lesssim \sum_j 2^{-|j-k|} A_j, \qquad\forall k\in\mathbb{Z}.
  \end{align*}
\end{lemma}

\begin{proof}
  For $z\in\mathbb{R}^n$, define
  \begin{align*}
    a_{z,k}(w) := S_ku(z) + \nabla S_k u(z) \cdot (w-z).
  \end{align*}
  From the representation in the lemma above, we have
  \begin{align*}
    \sup_{w\in B(z,r)} |u(w) - a_{z,k}(w)|
     & \le \frac{1}{2} r^2 \lVert  \nabla^2 S_k u \rVert_{\infty} + \sum_{j>k} \lVert P_j u \rVert_{\infty}      \\
     & \lesssim r^2 \sum_{j\le k} 2^{2j} \lVert P_j u \rVert_{\infty} + \sum_{j>k} \lVert P_j u \rVert_{\infty}.
  \end{align*}
  Setting $r=2^{-k}$, dividing by $r$ and taking $\sup_z$ yields the conclusion.
\end{proof}

\subsection{Closing the loop -- Littlewood--Paley bounds}\label{sec:absorption}
\begin{theorem}\label{thm:Rn-unit-ball}
  For each $n \ge 2$ and $0 < \alpha < 1$, there exists $\lambda_0 > 0$ such that the following holds.
  Let $f\colon B(0,1)\to\mathbb{R}$ be $L$-Lipschitz, $f(0)=0$, and $L<\lambda_0$.
  Suppose
  \[
    \wt\ve_f(z,r)\leq C(\alpha,\lambda_0)(C_{\varepsilon}+L) r^\alpha
    \quad\text{whenever } B(z,r)\subset B(0,1).
  \]
  Then $f\in C^{1,\alpha}(B(0,1/4))$.
\end{theorem}
\begin{proof}
  Fix $N>1$.
  Combining \cref{thm:Aj-upper-bound-bk,thm:bk-upper-bound-Aj} gives
  \begin{align}
    A_k
     & \le C_0 2^{-\alpha k}
    +C_1L^2\sum_{h\ge1}2^{-Nh}\mathfrak b_{k-h}
    \notag                                                \\
     & \le C_0 2^{-\alpha k}
    +C_2L^2\sum_{h\ge1}2^{-Nh}
    \sum_{\ell\in\mathbb{Z}}2^{-|k-h-\ell|}A_\ell  \notag \\
     & \le C_0 2^{-\alpha k}
    +C_3L^2\sum_{\ell\in\mathbb{Z}}2^{-|k-\ell|}A_\ell,
    \qquad k\in\mathbb{Z},
    \label{eq:closed-sequence-estimate}
  \end{align}
  since $|k-h-\ell|\ge |k-\ell|-h$.
  Here $C_0$ may depend on $C_\varepsilon$, whereas $C_1,C_2$
  are independent of $L$ and $C_\varepsilon$.

  Since $\alpha<\gamma<1$, we have
  \begin{align}
    \sum_{k\in\mathbb{Z}}2^{-\gamma|j-k|}2^{-\alpha k}
    \le C_{\alpha,\gamma}2^{-\alpha j}
    \quad\text{and}\quad
    \sum_{k\in\mathbb{Z}}2^{-\gamma|j-k|}2^{-|k-\ell|}
    \le C_\gamma 2^{-\gamma|j-\ell|},
    \label{eq:weighted-estimates}
  \end{align}
  where the latter follows from the triangle inequality
  in the form
  $\gamma|j-k|+|k-\ell|
    \ge \gamma|j-\ell|+(1-\gamma)|k-\ell|$.

  Set $\gamma=(1+\alpha)/2$ and define
  \[
    \mathcal A_j
    :=\sum_{m\in\mathbb{Z}}2^{-\gamma|j-m|}A_m \lesssim \sum_{m\in\mathbb{Z}} 2^{-\gamma|j-m|} \lVert u \rVert_{\mathrm{Lip}}<\infty.
  \]
  Fix $j\in\mathbb{Z}$. Multiply \cref{eq:closed-sequence-estimate}
  by $2^{-\gamma|j-k|}$ and sum in $k$.
  Interchanging nonnegative sums and applying
  \cref{eq:weighted-estimates}
  gives
  \begin{align*}
    \mathcal A_j
     & \le C_0\sum_{k\in\mathbb{Z}} 2^{-\gamma|j-k|}2^{-\alpha k}
    + C_3L^2\sum_{\ell\in\mathbb{Z}}A_\ell \sum_{k\in\mathbb{Z}} 2^{-\gamma|j-k|}2^{-|k-\ell|}                      \\
     & \le C_0C_{\alpha,\gamma}2^{-\alpha j}  +  C_3C_\gamma L^2  \sum_{\ell\in\mathbb{Z}}2^{-\gamma|j-\ell|}A_\ell \\
     & = C_0C_{\alpha,\gamma}2^{-\alpha j}  +  C_3C_\gamma L^2\mathcal A_j.
  \end{align*}

  Choosing $\lambda_0$ sufficiently small yields
  \[
    A_j\le\mathcal A_j
    \le 2C_0C_{\alpha,\gamma}2^{-\alpha j},
    \qquad j\in\mathbb{Z}.
  \]
  Consequently,
  \[
    \sup_{j\ge1}2^{(1+\alpha)j}\lVert P_ju \rVert_\infty
    =
    \sup_{j\ge1}2^{\alpha j}A_j
    <\infty.
  \]
  Together with $\lVert S_0u \rVert_\infty<\infty$,
  \cite[Theorem 1.4.9]{grafakosModernFourierAnalysis2014} gives
  $u\in C^{1,\alpha}(\mathbb{R}^n)$.
  Since $u=f$ on $B(0,1/4)$, the conclusion follows.
\end{proof}

\section{The general Ahlfors regular case}\label{secgeneral}

We will prove the following.

\begin{lemma}[\bf Main Lemma]\label{mainlemma}
  Let $\Omega\subset\R^{n+1}$ be a bounded open set with $n$-Ahlfors regular boundary and let $0<\alpha<1$. Suppose that there exists some constant $C$ such that
  $$\ve_n(x,r)\leq C\,r^{\alpha}\quad \mbox{ for all $x\in\pom$ and all $0<r\leq \diam(\pom)$.}$$
  Then, for $R>0$ small enough and for every $x\in\pom$, $\pom\cap B(x,R)$ is a Lipschitz graph (possibly rotated) with slope at most $\lambda_0$.
\end{lemma}

Putting together the previous lemma and \cref{teolipschitz}, one easily gets \cref{teomain}.

\vv
\subsection{Reifenberg flatness}

\begin{lemma}\label{lemreifemberg}
  Let $\Omega\subset\R^{n+1}$ be a bounded open set with $n$-Ahlfors regular boundary. Let $x\in\pom$ and $0<r\leq\diam(\pom)$.
  For any $\delta>0$, there exist constants $A>1$ and $0<\rho<1$ such that if
  $$\int_0^{Ar}\ve_n(y,t)^2\,\frac{dt}t\leq \rho\quad \mbox{ for all $y\in B(x,Ar)$, $0<t\leq Ar$,}$$
  then
  $$\beta_{\pom}(x,r)\leq \delta.$$
\end{lemma}

This follows from Lemma 6.1 in \cite{FTV2}.

\begin{lemma}
  Let $\Omega$ be a two-sided corkscrew open set with Ahlfors regular boundary. There exists $\epsilon_0>0$ depending on the two-sided corkscrew and Ahlfors regularity conditions such that, for $x\in\pom$, $0<r\leq \diam(\pom)$, if $\beta_{\pom}(x,r)\leq \epsilon$ with $\epsilon\leq \epsilon_0$, then
  $$b\beta_{\pom}(x,r/2)\leq C\epsilon.$$
\end{lemma}

\begin{proof}
  Let $L$ be an $n$-plane passing through $x$ such that $\dist(y,L)\leq 2\epsilon r$ for all $y\in B(x,r)$. Equivalently, $\pom\cap B(x,r)
    \subset F_L$, where $F_L$ is the closed $(2\epsilon r)$-neighborhood of $L$. We claim one of the components of $B(x,r)\setminus F_L$, which we denote by $V^+$, is
  contained in $\Omega^+$ and the other, which we denote by $V^-$, is contained in $\Omega^-$.

  Indeed, since $V^\pm$ are connected, $V^+$ is contained in
  either $\Omega^+$ or $\Omega^-$, and the same holds for $V^-$.
  By the two-sided corkscrew condition, there exist balls $B^\pm$
  contained in $\Omega^\pm\cap B(x,r)$ with radius $r(B^\pm)\geq c_0 r$. Assuming $2\epsilon\leq c_0/2$, it follows that both balls $B^\pm$ intersect $B(x,r)
    \setminus F_L$. So $V^+\cup V^-$ intersects both $\Omega^+$ and $\Omega^-$. This implies that one component, say $V^+$, is contained in $\Omega^+$, and the other, $V^-$, in $\Omega^-$.

  To see that $b\beta_\pom(x,r/2)\leq C\epsilon$, for a point $x\in L\cap B(x,r/2)$ let $S_x$ be a segment orthogonal to $L$ with length
  $5\epsilon r$, whose mid-point is $x$. Then one of the end-points is contained in $V^+$ and the other in $V^-$. Since $V^\pm\subset
    \Omega^\pm$, there exists a point $z_x\in S_x\cap\pom$. So
  $$\dist(x,\pom)\leq \frac12 \,\HH^1(S_x) = \frac52\,\epsilon r.$$
\end{proof}

\begin{coro}
  Let $\Omega\subset\R^{n+1}$ be a bounded open set with $n$-Ahlfors regular boundary and let $0<\alpha<1$. Suppose that there exists some constant $C$ such that
  $$\ve_n(x,r)\leq C\,r^{\alpha}\quad \mbox{ for all $x\in\pom$ and all $0<r\leq \diam(\pom)$.}$$
  Then, for any $\delta>0$, there exists $r_0>0$ depending on $\delta$ such that $b\beta_{\pom}(x,r)\leq \delta$ for all $x\in\pom$ and $0<r\leq r_0$.
\end{coro}

\begin{proof}
  This follows from the previous lemmas, taking into account that $\Omega$ is a two-sided corkscrew domain, by \cite{CTV}.
\end{proof}

Note that the above result implies that $\Omega$ is vanishing Reifenberg flat.

\subsection{The cube \texorpdfstring{$Q_0$}{Q₀} and the family \texorpdfstring{$\BA(Q_0)$}{BA(Q₀)}}

From now on, we work under the assumptions of \cref{mainlemma}.

We fix some constant $\lambda_0\in (0,1/2)$ as in \cref{teolipschitz}.
For some small $\epsilon\in (0,10^{-2})$ to be chosen below, with $\epsilon\ll \lambda_0$, we consider a radius $R_0>0$ small enough so that
\begin{equation}\label{eqbeta100}
  b\beta_{\pom}(x,r)\leq \epsilon\quad \mbox{ for all $x\in\pom$ and $0<r\leq\epsilon^{-1}R_0$.}
\end{equation}
We assume that $R_0$ is of the form $R_0=2^{-k_0}$ for some $k_0\in\Z$.

We denote by $\DD_\pom$ the David-Semmes lattice associated with $\pom$. For every $Q\in \DD_{\pom,k_0}$, we let $\BA(Q)$ be the family of maximal cubes $P\in\DD_\pom$ which are contained in $Q$ and satisfy
$$\angle (L_P,L_Q)>\frac1{10} \lambda_0,$$
where $L_P$ and $L_Q$ are minimizing hyperplanes for $\beta_\pom(3B_P)$ and $\beta_\pom(3B_Q)$, respectively.
For some $K\geq100$ to be fixed below, let $\cN_K(Q)$ be the family of cubes $R\in\DD_{\pom,k_0}$ such that $\dist(Q,R)\leq K\ell(Q)$ (in a sense, this should be considered as a large family of neighbors of $Q$).
We denote
$$\wt\BA(Q) = \bigcup_{R\in \cN_K(Q)}\BA(R).$$
Our objective is to show that $\wt\BA(Q)=\varnothing$ for every $Q\in \DD_{\pom,k_0}$. By standard arguments, this ensures that $\pom$ coincides with a
Lipschitz graph with slope $\leq \lambda_0/10$ in a neighborhood of $Q$, and then by \cref{teolipschitz} this implies
that $\pom$ is a $C^{1,\alpha}$ graph in a possibly smaller neighborhood of $Q$.

We will argue by contradiction, and we will assume that $\wt\BA(Q)\neq\varnothing$ for some $Q\in\DD_{\pom,k_0}$, which is equivalent to saying that
$\BA(Q)\neq\varnothing$ for some $Q\in\DD_{\pom,k_0}$. To this end, we consider the cube $Q_0 \in\DD_{\pom,k_0}$ which contains a cube $Q_1\in\BA(Q_0)$
with largest length $\ell_1$ among all the cubes from $\bigcup_{Q\in\DD_{\pom,k_0}}\BA(Q)$.

If $Q,R\in\DD_\pom$ and $R$ is a child of $Q$, from \cref{eqbeta100} one easily deduces that
$$\angle(L_Q,L_R) \leq C\epsilon.$$
From this fact, it follows also that the angle between any $n$-plane $L_Q$, with $Q\in \cN_K(Q_0)$, and $L_{Q_0}$ is bounded above
by $C(K)\epsilon$. Also,  for any cube $R\subset Q_0$, we deduce that
$$\angle(L_R,L_{Q_0})\leq C\epsilon\, |J(R)-k_0|,$$
where $J(R)$ is the index such that $R\in\DD_{\pom,J(R)}$.
In particular, if $Q_1\in \BA(Q_0)$ with $\ell_1=\ell(Q_1)=2^{-k_1}$,
$$ \log_2\frac{\ell(Q_0)}{\ell(Q_1)} = |k_1-k_0|\gtrsim \lambda_0\epsilon^{-1}.$$
Therefore,
$$\ell_1 \leq 2^{-c\lambda_0/\epsilon}\ell(Q_0).$$

We denote
$\wt Q_0 = \bigcup_{Q\in \cN_K(Q_0)}Q$ and we write $L_0:=L_{Q_0}$.
We let
$\tree(\wt Q_0)$ be the family of cubes from $\bigcup_{Q\in\cN_K(Q_0)} \DD_\pom(Q)$ which have side length at least equal to $\ell_1$.
By the construction of $\BA(\wt Q_0)$ and since, for any $Q\in \cN_K(Q_0)$, it holds that $\angle (L_Q,L_{Q_0})\leq C\epsilon\ll \lambda_0/10$ (assuming $\epsilon$ small enough, depending on $K$),
it follows that
$$\angle(L_P,L_0)|\leq \frac19\,\lambda_0\quad \mbox{ for any $P\in\tree(\wt Q_0)$.}$$

\vv
\section{The approximating Lipschitz graph}\label{sec-Lipgraph}

In this section we will construct a Lipschitz graph $\Gamma$ which will approximate $\pom$ in a neighborhood of $4B_{Q_0}$.
We will describe a construction based on the techniques of David and Semmes in \cite{DS1} in a more general context, also used in other contexts, such
as in connection with curvature of measures and rectifiability in \cite{Tolsa-llibre}. Since these techniques are well known, we will not write all the details of the proof, and instead we will refer the reader to \cite{DS1} and \cite{Tolsa-llibre}. If preferred, one could also apply the techniques of \cite{DT}.

After a suitable affine transformation, we can assume that the hyperplane $L_0$ is horizontal and passes through the origin.
We denote by $\Pi$ the orthogonal projection of $\R^{n+1}$ onto $L_0$.
We identify $L_0$ with $\R^n$ and we denote by $\DD_{L_0}$ the family of dyadic $n$-dimensional cubes of $L_0\equiv\R^n$.
We let
$\cW$ be the family of the cubes from $\DD_{L_0}$ of side length
$\ell_1$. Also, we let $\cW_0$ be the subfamily of the cubes from $\cW$
that intersect $\frac12KB_{Q_0}$. We index the families $\cW$ and $\cW_0$ by writing $\cW=\{S_i\}_{i \in I_\cW}$ and $\cW_0=\{S_i\}_{i \in I_0}$.

\begin{lemma}\label{l:whitney-prop1}
  Suppose that $K$ is large enough (depending only on $n$ and the Ahlfors regularity constant) and $\epsilon$ is small enough.
  Then, for every $i \in I_0$ there exists a cube $P_i \in \tree(\wt Q_0)$ such that
  $$
    \ell(S_i) \lesssim \ell(P_i) \lesssim \ell(S_i)$$
  and
  $$ \dist(S_i, \Pi(P_i)) \lesssim \ell(S_i).$$
\end{lemma}

This is proven in the same way as \cite[Lemma 7.21]{Tolsa-llibre}. We remark that in the constructions in \cite[Chapter 7]{Tolsa-llibre} and \cite[Chapter 8]{DS1} one needs to introduce some auxiliary Lipschitz regularizing functions $d,D:\R^{n+1}\to \R$.
In our context this is not necessary because all the stopping cubes from $\tree(\wt Q_0)$ have the same side length, and then one can take $D(x)=\ell_1$ in a  neighborhood of $\frac12KB_{Q_0}$.

For $i \in I_0$, denote by $A_i$ the affine function $L_0 \to L_0^\perp$ whose graph is the $n$-plane $L_{P_i}$ that minimizes
$\beta_\pom(3B_{P_i})$. Since  $P_i\in \tree(\wt Q_0)$, $\angle(L_{P_i}, L_0)\leq \lambda_0/2$, and thus $A_i$ is Lipschitz with constant at most $\tan(\lambda_0/10)\lesssim \lambda_0/5$.
On the other hand, for $i \in I_\cW\setminus I_0$, we put $A_i \equiv 0$.
We can now define $A$ on $L_0$.

Consider a smooth partition of unity $\{\phi_i\}_{i\in I_{\cW}}$ subordinate to $\{3S_i\}_{i\in I_{\cW}}$ in $L_0$, i.e., $\phi_i\in C^{\infty}_0(3S_i)$ with $\sum_i \phi_i\equiv 1$ on $L_0=\R^n \subset \R^{n+1}$.
We also require that, for every $i\in I_{\cW}$,
\begin{equation}\nonumber
  \|\nabla \phi_i \|_\infty \lesssim \ell_1^{-1} \quad\mbox{ and }\quad \|\nabla^2\phi_i\|_{\infty} \lesssim \ell_1^{-2}.
\end{equation}

Now,  we set
\begin{align*}
  A(p) := \sum_{i \in I_\cW} \phi_i(p) A_i(p)= \sum_{i \in I_0} \phi_i(p) A_i(p).
\end{align*}

We require the following lemma, which combines Lemmas 7.24 and 7.27 from \cite{Tolsa-llibre}.
\begin{lemma}\label{l:A}
  Suppose $K$ is large enough. Then
  the function $A: L_0 \to L_0^\perp$ is supported in $KB_{Q_0}\cap L_0$ and is Lipschitz with slope $\lesssim \lambda_0$ in $10B_{Q_0}$. Moreover, if $i\in I_{0}$, then for any $x \in 15 S_i$,
  \begin{align*}
    |\nabla^2A(x)| \lesssim \frac{\epsilon}{\ell_1}.
  \end{align*}
\end{lemma}

We will denote the graph of $A$ by $\Gamma$, that is
\begin{align}\label{e:graphA}
  \Gamma := \mathrm{Graph}(A)=\{ (x, A(x)) \, |\, x \in L_0\}.
\end{align}

\vv

The next three lemmas are similar to Lemmas 7.29, 7.30, and 7.32 from \cite{Tolsa-llibre}. For them to be valid, we need to assume $K$ large enough and $\epsilon$ small enough, as above.

\begin{lemma}\label{l:distQ-L}
  For $Q\in \tree(\wt Q_0)$ such that $Q\cap 10B_{Q_0}\neq\varnothing$ and $x\in \Gamma\cap 3B_Q$, it holds that
  \begin{align}\label{e:distQ-L}
    \dist(x,L_Q) \lesssim \epsilon \ell(Q).
  \end{align}
\end{lemma}

\begin{lemma}\label{l:distG-A}
  We have    \begin{align*}
    \dist(x, \Gamma) \lesssim \epsilon \ell_1\quad\text{for every }x\in \pom\cap 10B_{Q_0}.
  \end{align*}
\end{lemma}

\Cref{l:distG-A,l:distQ-L} combine to yield the following statement.

\begin{coro}\label{l:LBcloseGA}
  For $Q\in \tree(\wt Q_0)$  such that $Q\cap 10B_{Q_0}\neq\varnothing$, it holds that
  \begin{align}
    \dist(x, \Gamma)\lesssim \epsilon \ell(Q) \quad \text{for all }x\in L_Q\cap 2B_Q.
  \end{align}
\end{coro}

\begin{lemma}\label{l:GdistL0}
  We have
  $$\dist(x,L_0)\lesssim \,\epsilon\,\ell(Q_0)
    \quad \text{for all $x\in \Gamma$.}$$
\end{lemma}

\begin{coro}\label{coroGamma}
  For every $x\in\Gamma\cap 6B_{Q_0}$, we have
  \begin{equation}\label{eqbbetaG1}
    b\beta_\Gamma(x,r)\lesssim \epsilon\quad \mbox{ for all $r>0$,}
  \end{equation}
  and
  \begin{equation}\label{eqbbetaG2}
    b\beta_\Gamma(x,r)\lesssim \epsilon\,\frac r\ell_1\quad \mbox{ for all $r\leq\ell_1$.}
  \end{equation}
\end{coro}

\begin{proof}
  The estimate \cref{eqbbetaG2} for $r\leq\ell_1$ follows from the fact that $|\nabla^2A(\Pi(x))| \lesssim \epsilon/\ell_1$ for $x\in 6B_{Q_0}$.

  It suffices to prove \cref{eqbbetaG1} for $r\geq\ell_1$ because otherwise it follows from \cref{eqbbetaG2}.
  In the case $r\geq\ell_1$, the estimate is a consequence of \cref{l:distQ-L,l:GdistL0,l:LBcloseGA}.
\end{proof}
\vv

From the above results and the transversality between $\Gamma$ and $\partial(6B_{Q_0})$, we easily get the following.

\begin{coro}\label{coroguai}
  There exists a constant $C>0$ such that
  $$\pom\cap 6B_{Q_0}\subset U_{C\epsilon\ell_1} (\Gamma)$$
  and
  $$\Gamma\cap 6B_{Q_0}\subset U_{C\epsilon\ell_1} (\pom).$$
\end{coro}

Above, for $a>0$ and $F\subset\R^{n+1}$, $U_a(F)$ stands for the $a$-neighborhood of $F$.

\vv

\section{Proof of the Main Lemma}\label{secfinal}

\subsection{The main idea}

Let $A:L_0\to \R$ be the Lipschitz map constructed in \cref{sec-Lipgraph}. Consider a non-negative $C^\infty$ function
$\chi:L_0\to\R$ supported on $4B_{Q_0}\cap L_0$ which equals $1$ on $3B_{Q_0}\cap L_0$ and such that $\|\nabla^2\chi\|_\infty\lesssim 1/\ell(Q_0)^2\approx 1/R_0^2$.
Consider the function $u:L_0\to \R$ defined by
$$u= \chi A.$$
Denote, as in \cref{sec Littlewood-Paley},
$$
  \mathfrak b_k := \sup_{x_0\in\mathbb{R}^n} \mathfrak b_u (x_0, 2^{-k}).$$
Recall that we identify $L_0$ with $\R^n$.
We intend to show that
\begin{equation}\label{eqanglekey}
  \sum_{j=k_0}^{k_1} \mathfrak b_j \ll \lambda_0.
\end{equation}
Recall that $\ell(Q_0)=2^{-k_0}$ and $\ell_1=2^{-k_1}$. We claim that the last estimate leads to a contradiction. Indeed, for a ball $\Delta\subset L_0$, denote by $L_\Delta^u$ the
graph of the best affine approximation for $\mathfrak b_u(\Delta)$. Then,
for $Q\in\DD_{\pom,k}(Q_0)$ with parent $\wh Q$,
$$
  \angle(L_{\Pi(B_Q)}^u, L_{\Pi(B_{\wh Q})}^u) \lesssim \mathfrak b_u(\Pi(B_{\wh Q})) \leq \mathfrak b_{k-1}.$$
Hence,
$$\angle(L_{\Pi(B_{Q_1})}^u, L_{\Pi(B_{Q_0})}^u)\lesssim \sum_{j=k_0}^{k_1}\mathfrak b_{j}\ll\lambda_0.$$
However, by \cref{l:distQ-L} and the definition of $Q_1$, we have
$$\angle(L_{\Pi(B_{Q_1})}^u, L_{\Pi(B_{Q_0})}^u)\approx \lambda_0,$$
which yields the desired contradiction.
\vv

\subsection{Estimating \texorpdfstring{$\mathfrak b_k$}{𝔟ₖ}}

\begin{lemma}\label{lembkfac}
  The following holds:
  \begin{itemize}
    \item[(a)] $\mathfrak b_k \lesssim \epsilon$ for all $k\in\Z$.

    \item[(b)] $\mathfrak b_k \lesssim \epsilon 2^{-(k-k_1)}$ for all $k\geq k_1$.
  \end{itemize}
\end{lemma}

\begin{proof}
  The statement (b) follows from
  the fact that, by \cref{l:A},
  for any $x \in 15 S_i$ with $i\in I_{\cW}$, we have  $|\nabla^2u(x)| \lesssim \frac{\epsilon}{\ell(S_i)}=\frac{\epsilon}{\ell_1}$ (since
  $|\nabla^2A(x)| \lesssim \frac{\epsilon}{\ell(S_i)}$ and $\|\nabla^i\chi\|_\infty \lesssim \ell(Q_0)^{-i}$ for $i=1,2$).

  The statement (a) is due to \cref{thm:beta-graph-product*} and to the fact that for all $x\in \Gamma$ and all $r>0$, $b\beta_{\Gamma}(x,r)\lesssim
    \epsilon$. In the case $\ell_1\leq r\leq \ell(Q_0)$ this follows from \cref{l:distQ-L} and in the case $r\geq \ell(Q_0)$ from \cref{l:GdistL0}. The case  $0<r\leq \ell_1$ is a consequence of the statement (b).
\end{proof}

Recall that, by \cref{thm:bk-upper-bound-Aj},
\begin{equation}\label{eqPj1}
  \mathfrak b_k \lesssim \sum_j 2^{-|j-k|} 2^j \|P_j u\|_\infty.
\end{equation}
To estimate $\|P_j u\|_\infty$ we will use the following result.

\begin{lemma}\label{lempju}
  For $j\leq k_0$, it holds that
  $$\|P_ju\|_\infty \lesssim \epsilon\,2^{-k_0(n+1)+jn}.$$
  For every $j\geq k_0\geq1$ and $x_0\in L_0$, it holds that
  $$|  P_j u(x_0)|  \lesssim  2^{-2j+k_0} +   \left|\chi(x_0)\int \xi_{2^{-j}}(x_0-y_0) A(y_0) \,\mathrm{d}y_0  \right|.$$
\end{lemma}

\begin{proof}
  For $j\leq k_0$, using that $P_ju = \xi_{2^{-j}} * u$, where $*$ is the convolution in $\R^n$, we get
  $$\|P_ju\|_\infty \leq  \|\xi_{2^{-j}}\|_\infty \|u\|_{L^1(\R^n)} \leq \frac1{2^{-jn}}\,\epsilon \,2^{-k_0(n+1)}.$$

  On the other hand, for $j\geq k_0$, we argue as in \cref{thm:Aj-upper-bound-bk}, using that $\|\nabla^i\chi\|_\infty\lesssim 1/R_0^i$ for $i=1,2$, and that $\|u\|_\infty \lesssim \epsilon\,2^{-k_0}$, $\|\nabla u\|_\infty \lesssim \lambda_0\leq1$.
\end{proof}

Analogous to the calculation in \cref{sec:restricted-g}, one checks that for $x_0\in\supp \chi$, $A$ satisfies a Lipschitz condition at $x_0$, provided that $\lambda_0$ is sufficiently small.
Thus, we combine \cref{thm:LP-beta-general,lempju} to conclude
\begin{align*}
  |  P_j u(x_0)|  \lesssim  2^{-2j+k_0} + 2^{-j} \chi(x_0) a_{\rho,\Omega_A^+}(\vphi_A(x_0),2^{-j}) + \lambda_0^2 2^{-j} \chi(x_0) \sum_{h\ge 1} 2^{-Nh} \mathfrak b_A(x_0, 2^{h-j}),
\end{align*}
for $j\ge k_0 \ge 1$ and $x_0\in L_0$.

Also, by \cref{thm:beta-graph-product},
\begin{align*}
  \chi(x_0) \mathfrak b_A(x_0,r)
  \le \mathfrak b_{u}(x_0,r)  + \lVert A \rVert_{L^{\infty}(B(x_0,r))} \mathfrak b_\chi(x_0,r) + C\frac{r}{R_0}
  \lesssim  \mathfrak b_{u}(x_0,r)  + \frac{r}{R_0}.
\end{align*}
Thus, for $j\ge k_0$,
\begin{align*}
  |  P_j u(x_0)| & \lesssim_N \!  2^{k_0-2j} +  2^{-j}\chi(x_0) a_{\rho,\Omega_A^+}(\vphi_A(x_0),2^{-j})
  + \lambda_0^2 \,2^{-j} \sum_{h\ge 1} (\mathfrak b_u(x_0,2^{h-j}) + 2^{h-j+k_0})2^{-Nh}                 \\
                 & \lesssim_N 2^{k_0-2j} +  2^{-j}\chi(x_0) a_{\rho,\Omega_A^+}(\vphi_A(x_0),2^{-j})
  + \lambda_0^2 \,2^{-j} \sum_{h\ge 1} \mathfrak b_u(x_0,2^{h-j}) 2^{-Nh}.
\end{align*}
Plugging this into \cref{eqPj1}, we get
\begin{equation}\label{eqbk2}
  \mathfrak b_k \lesssim \sum_{j\geq k_0} 2^{-|j-k|} \Big(2^{k_0-j} + \!\sup_{x\in 5B_{Q_0}\cap\Gamma}\! a_{\rho,\Omega_A^+}(x,2^{-j})
  + \lambda_0^2   \sum_{h\ge 1} 2^{-Nh}\mathfrak b_{j-h}  \Big) + \epsilon\!\sum_{j<k_0}2^{-|j-k|}\, 2^{(j-k_0)(n+1)}.
\end{equation}
\vv

\begin{lemma}\label{lembkloop}
  We have
  \begin{equation}\label{eqsumk*}
    \sum_{k\geq k_0} \sum_{j\geq k_0} 2^{-|j-k|} \sum_{h\ge 1}2^{-Nh} \mathfrak b_{j-h}  \lesssim \sum_{j=k_0}^{k_1} \mathfrak b_j + \epsilon.
  \end{equation}
\end{lemma}

\begin{proof}
  We choose $N=1$.
  Then the term on the left-hand side of \cref{eqsumk*} equals
  \begin{align*}
    \sum_{k\geq k_0} \sum_{j\geq k_0} 2^{-|j-k|} \sum_{h<j}2^{-(j-h)} \mathfrak b_{h}
     & \leq \sum_h \mathfrak b_{h} \sum_{j\geq \max(k_0,h)} 2^{-(j-h)} \sum_{k\geq k_0} 2^{-|j-k|} \\
     & \lesssim
    \sum_h \mathfrak b_{h} \sum_{j\geq \max(k_0,h)} 2^{-(j-h)}                                     \\
     & \lesssim
    \sum_h \mathfrak b_{h}  2^{-(\max(k_0,h)-h)} = \sum_h \mathfrak b_{h}  2^{-\max(k_0-h,0)}
    .
  \end{align*}

  We split the last sum as follows:
  $$\sum_h \mathfrak b_{h}  2^{-\max(k_0-h,0)} =
    \sum_{h<k_0} \mathfrak b_{h}   2^{-(k_0-h)} +
    \sum_{k_0\leq h\leq k_1} \mathfrak b_{h}  + \sum_{h>k_1} \mathfrak b_{h} =: T_1+ T_2 + T_3.$$
  To estimate $T_1$ we use the fact that, by \cref{lembkfac}, $\mathfrak b_{h}\lesssim \epsilon$ for all $h\in\Z$, and we deduce that
  $$T_1\lesssim\epsilon.$$
  To deal with $T_3$, we take into account that, by \cref{lembkfac} again, $\mathfrak b_h \lesssim \epsilon 2^{-(h-k_1)}$ for all $h\geq k_1$. This gives
  $$T_3 = \sum_{h>k_1} \mathfrak b_{h} \lesssim \sum_{h>k_1}\epsilon 2^{-(h-k_1)}\lesssim\epsilon.$$
  Gathering the estimates above yields the lemma.
\end{proof}
\vv

Let $M_0>1$ be a constant to be fixed below.
Summing the estimate \cref{eqbk2} over $k_0+M_0\leq k\leq k_1$ and applying the previous lemma, we derive
\begin{align}\label{eqsum941}
  \sum_{k=k_0+M_0}^{k_1} \mathfrak b_k & \lesssim \sum_{k\geq k_0+M_0} \,\sum_{j\geq k_0} 2^{-|j-k|} 2^{k_0-j} +
  \sum_{k\geq k_0} \sum_{j\geq k_0} 2^{-|j-k|} \sup_{x\in 5B_{Q_0}\cap\Gamma} a_{\rho,\Omega_A^+}(x,2^{-j}) \nonumber \\
                                       & \quad +
  \epsilon\sum_{k\geq k_0}\sum_{j<k_0}2^{-|j-k|}2^{(j-k_0)(n+1)} + \lambda_0^2\sum_{j=k_0}^{k_1} \mathfrak b_j + \epsilon = S_1+\ldots +S_4 + \epsilon.
\end{align}
We estimate the first sum on the right-hand side:
\begin{equation}\label{eqfhjvi3}
  \sum_{k\geq k_0+M_0} \,\sum_{j\geq k_0} 2^{-|j-k|} 2^{k_0-j} =
  \sum_{k\geq k_0+M_0} \,\sum_{j\geq k} 2^{-(j-k)} 2^{k_0-j} + \sum_{k\geq k_0+M_0} \,\sum_{k_0\leq j <k} 2^{-(k-j)} 2^{k_0-j}.
\end{equation}
The first sum on the right-hand side equals
$$\sum_{k\geq k_0+M_0} 2^{k_0+k} \,\sum_{j\geq k} 2^{-2j} \lesssim  \sum_{k\geq k_0+M_0} 2^{k_0+k} \,2^{-2k} =
  \sum_{k\geq k_0+M_0} 2^{k_0-k} \approx 2^{-M_0}.$$
For the second sum on the right-hand side of \cref{eqfhjvi3}, we write
\begin{align*}
  \sum_{k\geq k_0+M_0} \,\sum_{k_0\leq j <k} 2^{-(k-j)} 2^{k_0-j} & =
  \sum_{k\geq k_0+M_0} 2^{k_0-k}\,\sum_{k_0\leq j <k} 1               \\ &=  \sum_{k\geq k_0+M_0} (k-k_0)\,2^{k_0-k} \leq  \sum_{k\geq k_0+M_0} 2^{(k_0-k)/2}  \approx 2^{-M_0/2}.
\end{align*}
Thus,
$$S_1\lesssim 2^{-M_0/2}.$$

Regarding the sum $S_2$, we have
$$S_2 =  \sum_{j\geq k_0} \sup_{x\in 5B_{Q_0}\cap\Gamma} a_{\rho,\Omega_A^+}(x,2^{-j}) \sum_{k\geq k_0} 2^{-|j-k|} \lesssim
  \sum_{j\geq k_0} \sup_{x\in 5B_{Q_0}\cap\Gamma} a_{\rho,\Omega_A^+}(x,2^{-j}).$$
For the term $S_3$, we write
\begin{align*}
  S_3 & = \epsilon\sum_{k\geq k_0}\sum_{j<k_0}2^{(j-k_0)(n+1)}2^{-(k-j)} =
  \epsilon \sum_{j<k_0}2^{(j-k_0)(n+1)}\sum_{k\geq k_0}2^{-(k-j)}                   \\
      & \lesssim \epsilon \sum_{j<k_0}2^{(j-k_0)(n+1)} 2^{j-k_0} \lesssim \epsilon.
\end{align*}

Gathering the estimates above, and using that $\mathfrak b_k\lesssim \epsilon$ for $k_0\leq k< k_0+M_0$,
we obtain
$$\sum_{k=k_0}^{k_1} \mathfrak b_k \lesssim M_0\epsilon + \sum_{k=k_0+M_0}^{k_1}
  \mathfrak b_k \lesssim M_0\epsilon +2^{-M_0/2} + \sum_{j\geq k_0} \sup_{x\in 5B_{Q_0}\cap\Gamma} a_{\rho,\Omega_A^+}(x,2^{-j}) + \lambda_0^2\sum_{j=k_0}^{k_1} \mathfrak b_j + \epsilon.$$
Assuming $\lambda_0$ small enough (depending at most on the Ahlfors regularity constant $C_0$), and choosing $M_0 = \lceil \epsilon^{-1/2}
  \rceil$ we get
\begin{equation}\label{eqbk***}
  \sum_{k=k_0}^{k_1} \mathfrak b_k \lesssim M_0\epsilon + 2^{-M_0/2} + \sum_{j\geq k_0} \sup_{x\in 5B_{Q_0}\cap\Gamma} a_{\rho,\Omega_A^+}(x,2^{-j}) + \epsilon \lesssim \epsilon^{1/2} + \sum_{j\geq k_0} \sup_{x\in 5B_{Q_0}\cap\Gamma} a_{\rho,\Omega_A^+}(x,2^{-j}).
\end{equation}

In \cref{lem7.11} below we will prove that
$$\sum_{j\geq k_0} \sup_{x\in 5B_{Q_0}\cap\Gamma} a_{\rho,\Omega_A^+}(x,2^{-j}) \lesssim  2^{-k_0\alpha} + \epsilon.
$$
Plugging this estimate into \cref{eqbk***} gives
$$
  \sum_{k=k_0}^{k_1} \mathfrak b_k \lesssim  2^{-k_0\alpha} + \epsilon^{1/2}.$$
Choosing $\epsilon$ small enough and $k_0$ large enough, we derive
$$\sum_{k=k_0}^{k_1} \mathfrak b_k \ll \lambda_0,$$
as desired, which concludes the proof of \cref{mainlemma}. In combination with \cref{teolipschitz}, we then get \cref{teomain}.

\vv
\subsection{Estimating \texorpdfstring{$a_{\rho,\Omega_A^+}(x,2^{-j})$}{a\_\{ρ,ΩA+\}(x, 2⁻ʲ)}}

\begin{lemma}\label{lemafc991}
  For all $x\in 5B_{Q_0}\cap\Gamma$ the following holds:
  \begin{itemize}
    \item[(a)] $a_{\rho,\Omega_A^+}(x,r)\lesssim \epsilon$ for all $r>0$.

    \item[(b)] $a_{\rho,\Omega_A^+}(x,r)\lesssim \epsilon\,\dfrac{r}{\ell_1}$ for all $0<r\leq \ell_1$.
  \end{itemize}
\end{lemma}

\begin{proof}
  For any $x\in 5B_{Q_0}\cap\Gamma$ and $t>0$, it is easy to check that
  $$\bigl| \mathcal{H}^{n}(\Omega_A^+\cap S(x,t)) - \frac{1}{2}\mathcal{H}^{n}(S(x,t)) \bigr|\lesssim b\beta_\Gamma(x,t)\, t^n.$$
  Then, by the definition of $a_{\rho,\Omega_A^+}$, we get
  \begin{align}
    a_{\rho,\Omega_A^+}(x,r)\label{eqavghj34}
     & = \frac{1}{r^{n+1}} \left|  \int_0^{\infty} \psi\Bigl(\frac{t}{r}\Bigr) \mathcal{H}^{n}(\Omega_A^+\cap S(x,t)) \,\mathrm{d} t -  \int_0^{\infty} \psi\Bigl(\frac{t}{r}\Bigr) \mathcal{H}^{n}(H_x^+\cap S(x,t)) \,\mathrm{d} t  \right| \notag \\
     & \le \frac{1}{r^{n+1}} \int_0^{\infty} \Bigl|  \psi\Bigl(\frac{t}{r}\Bigr) \Bigr| \bigl| \mathcal{H}^{n}(\Omega_A^+\cap S(x,t)) - \frac{1}{2}\mathcal{H}^{n}(S(x,t)) \bigr| \,\mathrm{d} t    \notag                                           \\
     & \lesssim \frac{1}{r^{n+1}} \int_0^{\infty} \Bigl|  \psi\Bigl(\frac{t}{r}\Bigr)\Bigr| \, b\beta_\Gamma(x,t)\, t^n \,\mathrm{d} t.
  \end{align}
  From \cref{eqbbetaG1}, it follows that, for $x\in 5B_{Q_0}\cap\Gamma$ and all $r>0$,
  $$ a_{\rho,\Omega_A^+}(x,r)
    \lesssim \epsilon\frac{1}{r^{n+1}} \int_0^{\infty} \Bigl|  \psi\Bigl(\frac{t}{r}\Bigr)\Bigr|\, t^n \,\mathrm{d} t \lesssim \epsilon.
  $$

  For the estimate in (b), we use \cref{eqbbetaG1,eqbbetaG2}:
  \begin{align*}
    a_{\rho,\Omega_A^+}(x,r) & \lesssim
    \frac{1}{r^{n+1}} \int_0^{\ell_1} \Bigl|  \psi\Bigl(\frac{t}{r}\Bigr)\Bigr| \, b\beta_\Gamma(x,t)\, t^n \,\mathrm{d} t +
    \frac{1}{r^{n+1}}\int_{\ell_1}^\infty \Bigl|  \psi\Bigl(\frac{t}{r}\Bigr)\Bigr| \, b\beta_\Gamma(x,t)\, t^n \,\mathrm{d} t \\
                             & \lesssim
    \frac{\epsilon}{r^{n+1}} \int_0^{\ell_1} \Bigl|  \psi\Bigl(\frac{t}{r}\Bigr)\Bigr| \, \frac{t^{n+1}}{\ell_1} \,\mathrm{d} t +
    \frac{\epsilon}{r^{n+1}} \int_{\ell_1}^\infty \Bigl|  \psi\Bigl(\frac{t}{r}\Bigr)\Bigr| \, t^n \,\mathrm{d} t.
  \end{align*}
  For the first integral we write
  $$
    \frac{\epsilon}{r^{n+1}} \int_0^{\ell_1} \Bigl|  \psi\Bigl(\frac{t}{r}\Bigr)\Bigr| \, \frac{t^{n+1}}{\ell_1} \,\mathrm{d} t
    \leq \frac{\epsilon}{r^{n+1}\ell_1} \int_0^{\infty} \Bigl|  \psi\Bigl(\frac{t}{r}\Bigr)\Bigr| \, t^{n+1} \,\mathrm{d} t
    \approx \epsilon \frac{r}{\ell_1}.
  $$
  For the second integral,
  $$\frac{\epsilon}{r^{n+1}}  \int_{\ell_1}^\infty \Big|\psi\Big(\frac tr\Big)\Big| \,t^{n}\,\dd t
    \lesssim \frac{\epsilon}{r^{n+1}}  \int_{\ell_1}^\infty \Big(\frac rt\Big)^{n+2} \,t^{n}\,\dd t
    \approx \epsilon \frac{r}{\ell_1}.
  $$
  Adding the last two estimates, (b) follows.
\end{proof}
\vv

\begin{lemma}\label{lemdifarho}
  Given $x\in 5B_{Q_0}\cap\Gamma$, let $x'\in\pom$ be such that $|x-x'|\lesssim\epsilon\ell_1$. For all $r\in[\ell_1,\ell(Q_0)]$,
  we have
  \begin{align}\label{eqx2}
    \big| a_{\rho,\Omega_A^+}(x,r) - a_{\rho,\Omega_A^+}(x',r)\big|\lesssim \frac{\epsilon \ell_1}r.
  \end{align}
\end{lemma}

\begin{proof}
  Write $A_{x,j}=A(x,2^j r,2^{j+1}r)$ for $j >0$ and $A_{x,0} = B(x,2r)$. Write also $A_{j}=A(0,2^{j-1},2^{j+2})$ for $j >0$ and $A_{0} = B(0,4)$.
  Using that $|x-x'|\leq \epsilon \ell_1\ll r$, it follows that
  $$\frac{|x-z|}r \in A_{j}  \quad \mbox{ if $z\in A_{x,j}$.}$$
  Then we get
  \begin{align*}
    \big|a_{\rho,\Omega_A^+} (x,r) - a_{\rho,\Omega_A^+} (x',r)\big| & \leq \frac1{r^{n+1}}\int_{\Omega_A^+}
    \bigg| \psi\biggl(\frac{|x-y|}{r}\biggr) - \psi\bigg(\frac{|x'-y|}{r}\bigg)\bigg|\,\dd y                                                      \\
                                                                     & \leq \sum_{j\geq0} \frac1{r^{n+1}}\int_{A_{x,j}} \sup_{z\in A_{x,j}}
    \bigg| \psi\biggl(\frac{|x-z|}{r}\biggr) - \psi\bigg(\frac{|x'-z|}{r}\bigg)\bigg|\,\dd y                                                      \\
                                                                     & \lesssim \sum_{j\geq0} \sup_{\xi\in A_j}|\nabla \psi(\xi)|\,\frac{|x-x'|}r
    \lesssim \frac{\epsilon\ell_1}r.
  \end{align*}
\end{proof}
\vv

\begin{lemma} \label{lemmaguai*}
  For all $x\in 5B_{Q_0}\cap\Gamma$ and all $r\in[\ell_1,\ell(Q_0)]$,  we have
  $$a_{\rho,\Omega_A^+}(x,r) \lesssim r^\alpha + \epsilon\,\frac r{\ell(Q_0)} +  \epsilon\,\frac{\ell_1}r.$$
\end{lemma}

\begin{proof}
  For $x$ and $r$ as in the lemma, let $x'\in\pom$ be such that  $|x-x'|\leq \epsilon \ell_1\ll r$, and consequently $x'\in 6B_{Q_0}$. By \cref{lemdifarho},
  \begin{align}\label{eqfhjc93}
    a_{\rho,\Omega_A^+}(x,r) \leq C\epsilon\,\frac{\ell_1}r + a_{\rho,\Omega_A^+}(x',r).
  \end{align}
  As in \cref{eqavghj34}, we have
  \begin{align*}
    a_{\rho,\Omega_A^+}(x',r) &
    \le \frac{1}{r^{n+1}} \int_0^{\infty} \Bigl|  \psi\Bigl(\frac{t}{r}\Bigr) \Bigr| \bigl| \mathcal{H}^{n}(\Omega_A^+\cap S(x',t)) - \frac{1}{2}\mathcal{H}^{n}(S(x',t)) \bigr| \,\mathrm{d} t \\
                              & =
    \frac{1}{r^{n+1}} \left(\int_0^{\ell_1}\cdots\; +  \int_{\ell_1}^{\ell(Q_0)}\cdots\; + \int_{\ell(Q_0)}^\infty\cdots\right)
    =: I_1 + I_2 +I_3.
  \end{align*}

  For the integral $I_1$ we use \cref{eqbbetaG1} and that $\|\psi\|_\infty\lesssim1$, and then we deduce
  $$I_1\lesssim \frac{1}{r^{n+1}} \int_0^{\ell_1} \Bigl|  \psi\Bigl(\frac{t}{r}\Bigr) \Bigr| t^{n-1}\epsilon\ell_1\,\mathrm{d} t
    \lesssim
    \frac{\epsilon}{r^{n+1}} \int_0^{\ell_1}  t^{n}\,\mathrm{d} t\lesssim
    \frac{\epsilon\,\ell_1^{n+1}}{r^{n+1}}.$$
  This follows from the fact that, for $r\leq\ell_1$, $\Gamma\cap S(x',r)$ is contained in a $C\ell_1$-neighborhood of the equator of $S(x',r)$.

  For $I_3$ we use again \cref{eqbbetaG1} and the fact that $|\psi(s)|\lesssim s^{-n-2}$:
  $$I_3\lesssim \frac{1}{r^{n+1}} \int_{\ell(Q_0)}^\infty \Bigl|  \psi\Bigl(\frac{t}{r}\Bigr) \Bigr| t^n\,b\beta(x',t)\,\mathrm{d} t
    \lesssim
    \frac{\epsilon}{r^{n+1}} \int_{\ell(Q_0)}^\infty\Big(\frac{r}{t}\Big)^{n+2}  t^{n}\,\mathrm{d} t\lesssim
    \frac{\epsilon\,r}{\ell(Q_0)}.$$
  To estimate $I_2$ we use \cref{coroguai} to deduce that, for $\ell_1\le t \le\ell(Q_0)$,
  \begin{align*}
    \bigl| \mathcal{H}^{n}(\Omega_A^+\cap S(x',t)) - \frac{1}{2}\mathcal{H}^{n}(S(x',t)) \bigr| &
    \leq \bigl| \mathcal{H}^{n}(\Omega^+\cap S(x',t)) - \frac{1}{2}\mathcal{H}^{n}(S(x',t)) \bigr|                                                   \\
                                                                                                & \quad +
    \mathcal{H}^{n}((\Omega^+ \Delta \Omega_A^+)\cap S(x',t))                                                                                        \\
                                                                                                & \lesssim \ve_n(x',t)\,t^n + \epsilon\ell_1 t^{n-1} \\& \lesssim t^{n+\alpha} + \epsilon\ell_1 t^{n-1}.
  \end{align*}
  Therefore,
  $$
    I_2 \lesssim \frac{1}{r^{n+1}}\int_0^\infty
    \Bigl|  \psi\Bigl(\frac{t}{r}\Bigr) \Bigr| \big(t^{n+\alpha} + \epsilon\ell_1 t^{n-1}\big)\,\mathrm{d} t
    \lesssim r^\alpha + \frac{\epsilon\ell_1}r.
  $$

  Gathering the estimates for $I_1$, $I_2$, and $I_3$ and using \cref{eqfhjc93}, the lemma follows.
\end{proof}

\vv

\begin{lemma}\label{lem7.11}
  We have
  $$\sum_{j\geq k_0} \sup_{x\in 5B_{Q_0}\cap\Gamma} a_{\rho,\Omega_A^+}(x,2^{-j}) \lesssim 2^{-k_0\alpha} + \epsilon.
  $$
\end{lemma}

\begin{proof}
  We denote
  $$I_1 = \{j:k_0\leq j\leq k_1\},\qquad I_2 = \{j:j> k_1\}.$$

  For the sum over $I_1$ we use \cref{lemmaguai*}.
  Recalling that $\ell_1=2^{-k_1}$, we get
  \begin{align*}
    \sum_{j\in I_1} \sup_{x\in 5B_{Q_0}\cap\Gamma} a_{\rho,\Omega_A^+}(x,2^{-j})
     & \lesssim \sum_{j\in I_1}
    \Big(2^{-j\alpha} + \epsilon\,2^{k_0-j} + \epsilon\,2^{j-k_1} \Big) \\
     & \leq \sum_{j\geq k_0} 2^{-j\alpha}
    +\epsilon \sum_{j\geq k_0} 2^{k_0-j} + \epsilon\sum_{j\leq k_1} 2^{j-k_1} \lesssim 2^{-k_0\alpha} + \epsilon.
  \end{align*}
  To estimate the sum over $I_2$ we use \cref{lemafc991}(b):
  $$\sum_{j\in I_2} \sup_{x\in 5B_{Q_0}\cap\Gamma} a_{\rho,\Omega_A^+}(x,2^{-j})
    \lesssim \sum_{j: 2^{-j}< \ell_1} \epsilon\,\dfrac{2^{-j}}{\ell_1} \lesssim \epsilon.$$
  Adding the two estimates above, the lemma follows.
\end{proof}

\printbibliography

\end{document}